\documentclass[11pt]{article}
\usepackage[a4paper,margin=28mm]{geometry}
\usepackage{amsmath,amssymb,amsthm,booktabs,array}
\usepackage[T1]{fontenc}
\usepackage{lmodern}
\usepackage{microtype}
\usepackage[hidelinks]{hyperref}
\newtheorem{theorem}{Theorem}
\newtheorem{lemma}[theorem]{Lemma}
\newtheorem{proposition}[theorem]{Proposition}
\theoremstyle{remark}\newtheorem{remark}[theorem]{Remark}
\newcommand{\R}{\mathbb R}
\newcommand{\tr}{\operatorname{tr}}

\newcommand{\cB}{\mathcal B}

\title{Nonexistence of strongly regular graphs via\\ multipoint spherical semidefinite bounds}
\author{Zhen-Qi Liao$^{1}$,\quad Wei-Hsuan Yu$^{1}$\\[6pt]
\small $^{1}$Department of Mathematics, National Central University, Taiwan\\[4pt]
\small \href{mailto:l0911487082@gmail.com}{l0911487082@gmail.com},\quad
\href{mailto:whyu@math.ncu.edu.tw}{whyu@math.ncu.edu.tw}}
\hypersetup{
 pdftitle={Nonexistence of strongly regular graphs via multipoint spherical semidefinite bounds},
 pdfauthor={Zhen-Qi Liao, Wei-Hsuan Yu}}
\date{}
\begin{document}
\maketitle
\begin{abstract}
We use multipoint semidefinite programming to prove the nonexistence
of strongly regular graphs. A normalized
eigenspace projection of a primitive strongly regular graph gives
a spherical two-distance set with one point for each vertex.
A bound smaller than the required number of points therefore rules
out the graph. Following the formulations of de Laat et al.
and Kao and Yu, we use, for each reference type, coefficient matrices
indexed by the feasible inner-product labels of the selected points.
In degree zero, we combine the matrix coordinates indexed by the
reference points into a single coordinate. The resulting programs use configurations
of up to six points, with matrix orders at most $2^m+1$ for a reference
set of $m$ points. Feasible dual solutions, verified in exact rational
arithmetic, exclude strongly regular graphs with parameters
$(351,140,73,44)$, $(550,162,75,36)$, $(703,182,81,35)$ and
$(1344,221,88,26)$.
\end{abstract}

\section{Introduction}
A strongly regular graph with parameters $(v,k,\lambda,\mu)$ is a simple
$k$-regular graph on $v$ vertices in which adjacent vertices have
$\lambda$ common neighbors and distinct nonadjacent vertices have
$\mu$ common neighbors. We ask whether a graph with prescribed
parameters can exist. Brouwer's online tables~\cite{BrouwerTable}
record candidate parameters, their spectra, and existence information.

Our starting point is the spherical embedding of a strongly regular
graph, also called its Euclidean representation
\cite[Section~1.3.5]{BVM}. For a primitive graph, meaning that
both the graph and its complement are connected, project the standard
basis vectors of $\mathbb R^v$ onto a nonprincipal eigenspace and
normalize the resulting vectors.
The resulting vectors lie on $S^{n-1}$, where $n$ is the eigenspace
dimension. Adjacent and nonadjacent vertices have inner products $b$
and $a$, respectively. If $r$ is the eigenvalue of this eigenspace,
then
\[
 b=\frac{r}{k},\qquad a=-\frac{r+1}{v-k-1}.
\]
Section~\ref{sec:spherical-reduction} derives these formulas.
These values are less than $1$, so distinct vertices give distinct
vectors. Thus the graph would give a spherical two-distance set of
exactly $v$ points. We use this implication as a necessary condition
for graph existence. To rule out the graph, it is enough to bound the
size of every spherical $\{a,b\}$-code in this dimension by less than $v$.
Here such a code is a finite set of unit vectors whose distinct
points have inner products in $\{a,b\}$; either or both values may occur.
The bound requires no graph regularity or prescribed common-neighbor counts.

Delsarte, Goethals and Seidel~\cite{DGS} obtain linear programming
(LP) bounds for spherical codes from inequalities involving Gegenbauer polynomials.
For the pair counts fixed by this spherical embedding, the degree-three
inequality is equivalent to the corresponding Krein condition.
The two nonprincipal eigenspaces give the two Krein conditions;
the calculation is included in Section~\ref{sec:spherical-reduction}.
The general absolute bound for a spherical two-distance set in
$\R^n$ is $n(n+3)/2$~\cite[Theorem~4.8 and Example~4.10]{DGS}.
If $a+b\ge0$, Musin's bound improves this to $n(n+1)/2$
\cite[Theorem~1]{MusinTwo}.

On the snapshot of Brouwer's tables accessed on 28 September 2026,
these classical spherical conditions exclude $286$ of the $466$
parameter entries marked as nonexistent with at most $1300$ vertices,
or $61.4\%$, counting a graph and its complement separately.
For this count, both nonprincipal eigenspaces are tested. A row is
excluded if either gives a negative cubic Gegenbauer sum, or if $v$
exceeds the absolute bound, or Musin's bound when $a+b\ge0$.
The snapshot and exact calculation are retained. This motivates
seeking stronger spherical bounds for the remaining cases.

For example, $(550,162,75,36)$ would give $n=33$ and
$\{a,b\}=\{-1/9,7/27\}$. Both $33(33+3)/2=594$ and
$33(33+1)/2=561$ exceed the required $550$ points.
Multipoint semidefinite programming (SDP) bounds give the following result.
\begin{theorem}\label{thm:main}
There is no strongly regular graph with any of the following parameters:
\[
\begin{gathered}
(351,140,73,44),\qquad (550,162,75,36),\\
(703,182,81,35),\qquad (1344,221,88,26).
\end{gathered}
\]
Their complementary parameter sets are also impossible.
\end{theorem}
These four parameter sets occur among the twelve cases listed as
unresolved in Koolen and Gebremichel~\cite[Table~1]{KG}.
We treat all twelve cases at levels $k=2,3,4,5,6$, with Gegenbauer
degrees $0,\ldots,5$ throughout.

We combine two multipoint SDP formulations. De Laat,
Caluza Machado, de Oliveira Filho and Vallentin~\cite{dL} give a
weakened Lasserre hierarchy whose dual collects each term according
to the union of its reference points and two selected points.
This supplies one rule for assembling the inequalities at every level.
Their polynomial formulation uses matrices indexed by monomials,
whose number grows with the coefficient degree. Kao and Yu
\cite[Theorem~2.13 and Corollary~2.14]{KY} use matrix constraints
indexed by inner-product values. For two-distance sets, their
degree-zero matrices have order at most $2^m+1$ and positive-degree
matrices have order at most $2^m$. Both works cover finite-distance
sets; their principal numerical applications concern equiangular lines.

We use de Laat et al.'s weakened dual and derive its matrix form from
Musin's positive semidefinite functions~\cite{Musin}.
First, we express the sphere coordinates in Musin's variables $u,v,t$:
the inner products of the two selected points with the reference
points, and with each other. For each reference type, we then store
the coefficient values in positive semidefinite (PSD) matrices indexed
by the allowed inner-product labels. Polynomial interpolation extends
these choices to functions with the required positivity. In degree
zero, the union inequalities use only sums over matrix coordinates
indexed by reference points. Combining those coordinates reduces
the matrix order from at most $2^m+m$ to $2^m+1$ for $m\ge1$,
without changing the feasible dual objective values.
These orders agree with those in Kao--Yu. The largest block needed
at six points has order $17$.

The counting bound applies to every finite $\Lambda\subset[-1,1)$.
For the coefficient-value construction, we assume $2\le k\le n$ and
that every reference configuration with at most $k-2$ points is
linearly independent. All twelve applications satisfy these assumptions.
The construction takes $(n,\Lambda,k,d)$ as input and produces fixed
coefficient matrices and one SDP whose unknowns are PSD matrices and
a cardinality bound.

The coordinate change describes the same functions using $t,u,v$.
Fixing a largest Gegenbauer degree $d$ restricts the search and may
weaken the bound. Within this polynomial class, with no degree limit
on the coefficient polynomials in $u,v$, the label-indexed matrices
describe exactly the coefficient values needed by the code constraints. The subsequent
degree-zero compression preserves the union inequalities and their
feasible objective values, but need not preserve individual function
values. Thus the construction gives an explicit way to assemble
de Laat et al.'s inequalities using the matrix sizes of Kao--Yu,
together with interpolation and lifting proofs and exact certificates
for the graph exclusions. A verified feasible solution proves the
bound; it need not be optimal.

Section~\ref{sec:coordinates} introduces the weakened dual as a
cardinality bound and derives its expression in the inner-product
variables $u,v,t$. Section~\ref{sec:framework} performs three reductions: it replaces the
coefficient functions by their values on the feasible labels, removes
reference-order redundancy, and compresses the unused degree-zero
coordinates before stating the final SDP. Section~\ref{sec:application}
applies the formulation to strongly regular graphs, first deriving the
spherical embedding and then checking and solving the twelve models used
here. Section~\ref{sec:workflow} describes the exact certificate test,
the exact model construction, the numerical search, and independent
cross-checks. Appendix~\ref{app:reproduce} gives the public repository
for the programs and the four certificates proving Theorem~\ref{thm:main}.

Throughout, $\Lambda$ is the set of allowed inner products, $k$ is
the multipoint level, and $d$ is the largest Gegenbauer degree retained.
Thus $d$ bounds the degree in $t$, not the coefficient degrees in $u,v$.
Appendix~\ref{app:degree} uses $\delta$ for the monomial degree parameter.
In graph parameter tuples and the spectral calculation, $k$ instead
denotes valency; its meaning is specified where it changes.

\section{Preliminaries on semidefinite bounds for spherical codes}\label{sec:coordinates}
This section recalls the multipoint dual bound of de Laat
et al.~\cite{dL} and Musin's construction of positive semidefinite
functions~\cite{Musin}. We express these ingredients in inner-product
variables and fix the notation used in Section~\ref{sec:framework}.

We first state the cardinality bound, use orthogonal symmetry to choose
one function for each reference type, and express these functions in
inner-product variables. We then construct polynomial candidates using
Musin's polynomials and impose the remaining reference symmetries.
The coordinate change describes the same functions; the later
polynomial restriction may weaken the best obtainable bound.

\subsection{A multipoint bound for spherical codes}
Let $\Lambda\subset[-1,1)$ be finite and nonempty, and fix an integer
$k\ge2$. The value of $k$ specifies how many points may occur in one
constraint. Let $V=S^{n-1}\subset\mathbb R^n$. A spherical
$\Lambda$-code is a finite subset of $V$ whose distinct points have
inner products in $\Lambda$. For $j\ge0$, let $I_j$ be the collection
of such codes of size at most $j$, including the empty set, and write
$I_{=j}$ for those of size exactly $j$. A member of $I_j$ will also be
called a configuration.

We seek a real-valued continuous function
$T:V^2\times I_{k-2}\to\mathbb R$, symmetric in its two sphere
arguments: $T(x,y,Q)=T(y,x,Q)$. The meaning of continuity in the
reference configuration $Q$ is specified below. For every fixed
$Q\in I_{k-2}$, we require the following positivity property:
\[
 \bigl(T(x_i,x_j,Q)\bigr)_{i,j=1}^N\succeq0
 \quad\text{for every finite list }x_1,\ldots,x_N\in S^{n-1}.
\]
For a real symmetric matrix $M$, the notation $M\succeq0$ means
$c^TMc\ge0$ for all real vectors $c$;
$M\succ0$ means that the inequality is strict for $c\ne0$.
We call a function with this property positive semidefinite (PSD),
using the positivity convention of~\cite{dL}.
The property is required for all sphere points, even when their
mutual inner products do not belong to $\Lambda$.

The PSD property gives a nonnegative sum for each fixed reference
configuration: if $C$ is a code, then $\sum_{x,y\in C}T(x,y,Q)$ is
the quadratic form of a PSD matrix at the all-ones vector.
We sum these expressions over all reference subsets $Q\subseteq C$
with $|Q|\le k-2$ and bound the result from above.
Each term uses at most $k$ points, namely the points of
$S=Q\cup\{x,y\}$. Grouping the terms by this union leads to
\begin{equation}\label{eq:B}
 \cB_kT(S)=
 \sum_{\substack{Q\subseteq S\\|Q|\le k-2}}
 \ \sum_{\substack{(x,y)\in S^2\\Q\cup\{x,y\}=S}}T(x,y,Q)
 \qquad(S\in I_k\setminus\{\varnothing\}).
\end{equation}
Thus $\cB_k$ collects all terms with the same union $S$.
Pairs are ordered, and the selected points may coincide or lie in $Q$.
For example, if $k=3$ and $S=\{x,y\}$ with $x\ne y$, symmetry in
$x,y$ gives
\[
\begin{aligned}
 \cB_3T(\{x,y\})={}&2T(x,y,\varnothing)\\
 &+T(y,y,\{x\})+2T(x,y,\{x\})\\
 &+T(x,x,\{y\})+2T(x,y,\{y\}).
\end{aligned}
\]
The three lines correspond to $Q=\varnothing,\{x\},\{y\}$.

For a proposed bound $\alpha=1+\lambda$, we require, for every
nonempty $S\in I_k$,
\[
 \cB_kT(S)\le
 \begin{cases}
 \lambda,&|S|=1,\\
 -2,&|S|=2,\\
 0,&3\le |S|\le k.
 \end{cases}
\]
To relate this problem to the formulation of de Laat et al., consider the graph
$\mathcal G_\Lambda$ on $V$ whose edges join distinct points with
inner product outside $\Lambda$. Its independent sets are exactly our codes.
In the notation of de Laat et al.~\cite[equation~(9)]{dL}, the
weakened dual is
\begin{equation}\label{eq:delaat4}
 \Delta_k(\mathcal G_\Lambda)^*
 =\inf\left\{1+\lambda:
 \begin{array}{l}
 \lambda\in\mathbb R,\quad
 T\in C(V^2\times I_{k-2})_{\succeq0},\\
 \cB_kT\le\lambda\chi_{I_{=1}}-2\chi_{I_{=2}}
 \end{array}\right\}.
\end{equation}
Here $\chi$ denotes an indicator function, and
$C(V^2\times I_{k-2})_{\succeq0}$ means the continuous symmetric
functions with the PSD property just stated.
Continuity in the reference configuration is understood as follows:
configurations converge when their sizes are eventually equal and
their points can be ordered to converge coordinatewise. The empty
configuration is isolated. The next construction respects this continuity.

\begin{proposition}\label{prop:certificate}
Every feasible pair $(\lambda,T)$ in~\eqref{eq:delaat4} gives
$|C|\le1+\lambda$ for every spherical $\Lambda$-code $C$.
Equivalently, with $\alpha=1+\lambda$, it gives $|C|\le\alpha$.
\end{proposition}
\begin{proof}
Let $N=|C|>0$. Sum the nonnegative quadratic forms just described,
and then group the terms by their unions:
\begin{align*}
 0&\le\sum_{\substack{Q\subseteq C\\|Q|\le k-2}}
             \sum_{x,y\in C}T(x,y,Q)\\
  &=\sum_{\substack{\varnothing\ne S\subseteq C\\|S|\le k}}
             \cB_kT(S)\\
  &\le N\lambda-2\binom N2=N(1+\lambda-N).
\end{align*}
Every triple $(Q,x,y)$ occurs once, in the term for
$S=Q\cup\{x,y\}$. There are $N$ singleton subsets and $\binom N2$
two-point subsets; all larger subsets contribute at most zero.
Division by $N$ proves the claim. On a singleton, $\cB_kT$ is a sum
of diagonal entries $T(x,x,Q)\ge0$, so feasibility also implies
$\lambda\ge0$; hence the empty code satisfies the bound.
\end{proof}

For a nonexistence proof, a feasible function suffices; neither
optimality nor attainment of the infimum is required. We next construct
functions with the required positivity and symmetry.

\subsection{One function for each reference type}
At first, $T$ appears to require a separate function for every reference
configuration $Q$. Orthogonal symmetry reduces this choice to one
function for each reference shape.
Orthogonal transformations preserve the allowed inner products and
the union condition in~\eqref{eq:B}. By the symmetry reduction of
de Laat et al.~\cite[Section~3]{dL}, restricting the dual to functions
satisfying
\begin{equation}\label{eq:global-invariance}
T(gx,gy,gQ)=T(x,y,Q)\qquad(g\in O(n))
\end{equation}
does not change its feasible objective values. Here $O(n)$ denotes the orthogonal
transformations of $\mathbb R^n$.

Two reference sets have the same \emph{reference type}, or shape,
if an orthogonal transformation carries one to the other.
We can test this using Gram matrices.
If two ordered lists
$(a_i)$ and $(b_i)$ have the same Gram matrix, then
$\sum_i c_i a_i\mapsto\sum_i c_i b_i$ is well-defined and preserves
inner products. Indeed, the squared norm of either linear combination
is the same quadratic form in their common Gram matrix. Extending
orthonormal bases of the two spans gives an orthogonal transformation
carrying each $a_i$ to $b_i$.
Consequently, an unordered reference configuration is determined up to
an orthogonal transformation by its Gram matrix up to simultaneous
row and column permutation. We refer to a Gram matrix up to these permutations as a \emph{Gram type}.
For $m$ reference points the diagonal entries are $1$ and the
other entries belong to the finite set $\Lambda$, so only finitely
many types occur. A matrix is realizable in $\mathbb R^n$ precisely
when it is PSD and has rank at most $n$. The condition
$1\notin\Lambda$ ensures that its distinct indices represent distinct points.

Choose one concrete reference configuration $R\subset S^{n-1}$
from each type with at most $k-2$ points, and let
$\mathcal R_{k-2}\subseteq I_{k-2}$ be this finite collection.
For each $R$, put $m=|R|$, fix an ordering $r_1,\ldots,r_m$ of its
points, and write $G_R=(r_i\cdot r_j)_{i,j=1}^m$ for its Gram matrix.
Thus $R$ is a set of sphere points, while $G_R$ is the matrix for
its chosen order. It is enough to store the functions
\[
 K_R(x,y)=T(x,y,R).
\]
For another configuration $Q$, let $R_Q$ be its representative and
choose an orthogonal map $\gamma_Q$ carrying $Q$ to $R_Q$. Then
\begin{equation}\label{eq:local-assembly}
 T(x,y,Q)=K_{R_Q}(\gamma_Qx,\gamma_Qy).
\end{equation}
Thus we move the reference set to its representative and evaluate the
corresponding function.

There may be several choices of $\gamma_Q$. To obtain a well-defined
value, we require
\[
 K_R(gx,gy)=K_R(x,y)\quad\text{whenever }gR=R.
\]
This requirement follows from~\eqref{eq:global-invariance}.
It is also sufficient: if $\gamma_Q'$ is another choice, then
$\gamma_Q'\gamma_Q^{-1}$ preserves $R_Q$, and hence gives the same
value in~\eqref{eq:local-assembly}.

Conversely, continuous PSD functions $K_R$ with this symmetry give a
valid continuous PSD function $T$ by~\eqref{eq:local-assembly}.
For positivity, evaluate $K_{R_Q}$ on the transformed list of sphere
points; its matrix is exactly the required matrix of $T$ values.
Joint continuity follows from~\cite[Proposition~3.1]{dL}, since there
are finitely many reference types and each $K_R$ is continuous and
invariant under the transformations preserving $R$. No continuous
choice of $\gamma_Q$ is required; Appendix~\ref{app:assembly-continuity}
gives a direct sequence argument.

\subsection{Expressing the functions in inner-product variables}\label{sec:inner-product-coordinates}
Fix a reference representative $R$. To express $K_R$ using inner
products, use the chosen order $R=(r_1,\ldots,r_m)$ and put
\[
 A_R=[r_1\ \cdots\ r_m],\qquad G_R=A_R^TA_R,
 \qquad u=A_R^Tx,\quad v=A_R^Ty,\quad t=x\cdot y.
\]
Here $A_R$ is an $n\times m$ matrix, $G_R$ is an $m\times m$ matrix,
$u,v\in\mathbb R^m$, and $t\in\mathbb R$.
Thus $u_i=r_i\cdot x$ and $v_i=r_i\cdot y$. Together with $G_R$, the tuple $(t,u,v)$ records every inner product
among the reference points and the two selected points.
We call $u$ and $v$ \emph{inner-product labels}. A label records
only the inner products with the reference points, so distinct selected
points can have the same label. Even when $u=v$, the value of
$t=x\cdot y$ is still needed; for unit vectors, $x=y$ holds exactly
when $t=1$.
For example, with one reference point $r$, these data are simply
$r\cdot x$, $r\cdot y$, and $x\cdot y$, and the joint Gram matrix is
\[
 \begin{pmatrix}1&u&v\\u&1&t\\v&t&1\end{pmatrix}.
\]
No coordinates in $\mathbb R^n$ are needed to describe this three-point shape.

If two pairs $(x,y)$ and $(x',y')$ give the same tuple, the ordered
lists $(r_1,\ldots,r_m,x,y)$ and $(r_1,\ldots,r_m,x',y')$ have the
same Gram matrix. The preceding argument supplies an orthogonal map
fixing each $r_i$ and carrying one pair to the other. Since $K_R$
does not change under such maps, it has a well-defined expression
\begin{equation}\label{eq:local-coordinates}
 K_R(x,y)=F_R(t,u,v)
         =F_R(x\cdot y,A_R^Tx,A_R^Ty).
\end{equation}
The domain of $F_R$ is the set
\[
 \{(x\cdot y,A_R^Tx,A_R^Ty):x,y\in S^{n-1}\}.
\]
Here $x,y$ range over all sphere points; they need not belong to a
code. On this domain $F_R$ is continuous: for a convergent sequence of tuples,
choose realizing pairs. Compactness of the sphere supplies convergent
subsequences of pairs; their limits realize the limiting tuple and
all give the same value of $K_R$. This proves convergence of the
$F_R$ values. Interchanging $x$ and $y$ gives
$F_R(t,u,v)=F_R(t,v,u)$.

The positivity condition is unchanged. For
any finite list $x_1,\ldots,x_N$, we have an equality of matrices
\[
 \bigl(K_R(x_i,x_j)\bigr)_{i,j}
 =\bigl(F_R(x_i\cdot x_j,A_R^Tx_i,A_R^Tx_j)\bigr)_{i,j}.
\]
Let $\operatorname{PD}(S^{n-1},R)$ denote the continuous real-valued
functions $F$ on the displayed inner-product domain that satisfy
$F(t,u,v)=F(t,v,u)$ and the PSD evaluation condition above.
This is the function class used in Musin's construction
(``positive definite'' here means PSD, allowing zero eigenvalues).
Conversely, any function in this class gives a continuous PSD function
of $x,y$ by the same substitution. Invariance under permutations
preserving $G_R$ is an additional requirement for an unordered reference
set; it will be imposed in Section~\ref{sec:reference-order}.
No polynomial assumption or degree restriction has been made in this
coordinate change.

\subsection{Constructing PSD functions from matrices using Musin's polynomials}
We now construct polynomial PSD functions using Musin's polynomials
and PSD coefficient matrices. The construction first requires the
lengths and inner products of the components perpendicular to the
reference span.
Assume from here on that $2\le k\le n$ and that every reference
configuration in $I_{k-2}$ is linearly independent, as
in~\cite[Section~4]{dL}.
We verify this assumption for our applications below.
In particular, $m\le k-2\le n-2$, so $G_R$ is positive definite and
the perpendicular space has dimension at least two.

The projection of $x$ onto the reference span is $A_RG_R^{-1}u$:
it lies in that span and has the same inner products with every
$r_i$ as $x$ does. Thus set
\[
 x_\perp=x-A_RG_R^{-1}u,\qquad y_\perp=y-A_RG_R^{-1}v.
\]
Since $A_R^Tx_\perp=A_R^Ty_\perp=0$, these are the perpendicular
components. Their squared lengths and inner product are
\begin{equation}\label{eq:projection-identities}
 \begin{split}
 h_u:=\|x_\perp\|^2&=1-u^TG_R^{-1}u,\qquad
 h_v:=\|y_\perp\|^2=1-v^TG_R^{-1}v,\\
 w:=x_\perp\cdot y_\perp&=t-u^TG_R^{-1}v.
 \end{split}
\end{equation}
For an empty reference set, all reference matrices and labels are
empty, their products are zero, and $x_\perp=x$, $y_\perp=y$.

The same decomposition gives the feasibility test for labels and
triples used in Section~\ref{sec:framework}.

\begin{lemma}\label{lem:framework-gram}
For an independent reference set of size $m\le n-2$, a vector $u$
is realized as the inner-product label of a unit vector if and only if
\[
 \begin{pmatrix}G_R&u\\u^T&1\end{pmatrix}\succeq0,
 \quad\text{equivalently }1-u^TG_R^{-1}u\ge0.
\]
A triple $(t,u,v)$ is realized by two unit vectors if and only if
\[
 H_{u,v,t}=\begin{pmatrix}
 G_R&u&v\\u^T&1&t\\v^T&t&1
 \end{pmatrix}\succeq0.
\]
\end{lemma}
\begin{proof}
Realized vectors give these Gram matrices, proving necessity.
For the converse, the following invertible change of variables separates
the reference components from the perpendicular components:
\[
 C^TH_{u,v,t}C=
 \begin{pmatrix}G_R&0&0\\0&h_u&w\\0&w&h_v\end{pmatrix},
 \qquad
 C=\begin{pmatrix}
 I_m&-G_R^{-1}u&-G_R^{-1}v\\
 0&1&0\\0&0&1
 \end{pmatrix}.
\]
One matrix is PSD if and only if the other is, since $C$ is invertible.
If they are PSD, the lower $2\times2$ block is a Gram matrix of two
vectors $x_\perp,y_\perp$ in the perpendicular space, whose dimension is
at least two. Then
\[
 x=A_RG_R^{-1}u+x_\perp,\qquad y=A_RG_R^{-1}v+y_\perp
\]
have the required labels, unit norms, and inner product $t$.
The one-vector assertion uses a perpendicular vector of squared
norm $h_u$. The condition $m+2\le n$ ensures enough space for every
Gram matrix in this argument.
\end{proof}

To compare these formulas with the orthonormal coordinates in Musin
and de Laat, factor $G_R=BB^T$, with $B$ lower triangular and positive
diagonal, and set
\[
 E_R=A_RB^{-T},\qquad L_R(u)=B^{-1}u,\qquad
 L(A_R)=B^{-1}A_R^T.
\]
The columns of $E_R$ are an orthonormal basis of the reference span,
since $E_R^TE_R=I_m$. The map $L_R$ converts an inner-product label
into coordinates in this basis, while $L(A_R)=E_R^T$ takes a sphere
vector to the same coordinates. Thus the projection coordinates used
in those papers are
\[
 \xi=L(A_R)x=B^{-1}u,\qquad \eta=L(A_R)y=B^{-1}v.
\]
Their norms are at most one. The labels $u,v$ themselves need not
have norms at most one, because the reference points need not be
orthogonal. The factor $B$ explains the relation between the two
coordinate systems. The formulas used in the computation depend only
on the rational expressions in the Gram and label entries
in~\eqref{eq:projection-identities}, without computing $B$ or square roots.

We can now use these projection quantities to construct functions
whose evaluation matrices are PSD. Put $p=n-m$, the dimension of
the perpendicular space. For each
nonnegative integer $l$, let $P_l^p$ be the Gegenbauer polynomial
of degree $l$, normalized by $P_l^p(1)=1$. Applying $P_l^p$ to the
inner products of unit vectors in this space gives a PSD matrix.
For $p=2$, use the Chebyshev polynomials of the first kind.
Musin applies these polynomials to the normalized perpendicular
components and then multiplies back their lengths:
\[
 P_l^{n,m}(t,\xi,\eta)=
 \bigl((1-\|\xi\|^2)(1-\|\eta\|^2)\bigr)^{l/2}
 P_l^{n-m}\!\left(
 \frac{t-\xi\cdot\eta}
 {\sqrt{(1-\|\xi\|^2)(1-\|\eta\|^2)}}\right).
\]
The fraction is the inner product of the normalized perpendicular
vectors. The factor outside the polynomial is the product of their lengths to the power $l$.
Initially suppose these lengths are nonzero. In our labels, write
the same expression as
\begin{equation}\label{eq:gram-gegenbauer}
 P_l^{n,m}(G_R;u,v,t)
 :=P_l^{n,m}(t,B^{-1}u,B^{-1}v).
\end{equation}
\begin{samepage}
Using~\eqref{eq:projection-identities}, the apparent divisions cancel.
For the fixed dimension $p=n-m$,
write the Gegenbauer polynomial as
$P_l^p(s)=\sum_{j=0}^{\lfloor l/2\rfloor}a_{l,j}s^{l-2j}$,
where the coefficients $a_{l,j}$ depend on $p$. Only powers with the
same parity as $l$ occur. Recall that
\[
 h_u=1-u^TG_R^{-1}u,\qquad
 h_v=1-v^TG_R^{-1}v,\qquad w=t-u^TG_R^{-1}v.
\]
Substitution gives
\[
 P_l^{n,m}(G_R;u,v,t)
 =\sum_{j=0}^{\lfloor l/2\rfloor}
       a_{l,j}w^{l-2j}(h_uh_v)^j.
\]
\end{samepage}
This polynomial formula also defines the value when either
perpendicular component is zero. Suppressing the common arguments
$(G_R;u,v,t)$ in the next two displays, the polynomials of degrees
zero, one and two are
\begin{equation}\label{eq:kernels}
 P_0^{n,m}=1,\qquad P_1^{n,m}=w,\qquad
 P_2^{n,m}=\frac{p w^2-h_uh_v}{p-1}.
\end{equation}
In particular, the degree-one matrix is just the Gram matrix of the
perpendicular components. Starting from $P_0^{n,m}$ and $P_1^{n,m}$,
all further degrees are evaluated by
\begin{equation}\label{eq:homogeneous-recurrence}
 P_l^{n,m}=\frac{(2l+p-4)wP_{l-1}^{n,m}
                    -(l-1)h_uh_vP_{l-2}^{n,m}}{l+p-3}
 \qquad(l\ge2).
\end{equation}
The denominators are positive because $p\ge2$.

For every $l$, these functions have PSD evaluation matrices on any
finite list of sphere points. To see this, normalize the nonzero
perpendicular components and take their Gegenbauer evaluation matrix,
which is PSD. Multiplying row and column $i$ by $\|(x_i)_\perp\|^l$
preserves PSD and gives the displayed values. A zero component gives
a zero row for $l>0$; at $l=0$ every entry is one.
This is Musin's construction~\cite[Theorem~3.1]{Musin}.

We now restrict the class of functions used in the search. Fix a
nonnegative integer $d$ and consider polynomials $F_R$ in all the
variables $(t,u,v)$ whose restrictions to the sphere-point domain
belong to $\operatorname{PD}(S^{n-1},R)$ and whose degree in $t$ is at
most $d$. We use the same notation for such a polynomial and its
restriction to that domain. No degree bound is imposed on its
coefficients as polynomials in $u,v$.

Unlike the coordinate change, this restriction may weaken the best
bound. We search directly in the restricted class and require all the
original inequalities in~\eqref{eq:delaat4}; feasibility is not inferred
by truncating another function. Every feasible choice still gives a
valid bound by Proposition~\ref{prop:certificate}.

A symmetric polynomial $f(u,v)$ is called a \emph{PSD polynomial}
if it has a representation
\[
 f(u,v)=z(u)^THz(v),\qquad H\succeq0,
\]
where $z$ is a finite column vector of real polynomials and $H$ is a
constant real symmetric matrix of matching order. This is the
convention of Musin~\cite[Definition~2.3]{Musin}. For any finite list of labels,
the evaluation matrix is PSD, since it is of the form $Z^THZ$, where
the columns of $Z$ are the corresponding values of $z$. This condition
does not require every scalar value $f(u,v)$ to be nonnegative.

Musin's polynomial theorem~\cite[Theorem~3.2]{Musin} gives
\begin{equation}\label{eq:musin-expansion}
 F_R(t,u,v)=\sum_{l=0}^d
       f_{R,l}(u,v)P_l^{n,m}(G_R;u,v,t),
 \qquad f_{R,l}\text{ a PSD polynomial}.
\end{equation}
Here the scalar factors $P_l^{n,m}(G_R;u,v,t)$ are fixed by $n,R,l$.
The choices are the coefficient polynomials $f_{R,l}(u,v)$, which
do not depend on $t$. They are functions of the labels, not constant scalars.
To apply the theorem in our coordinates, first apply it to
$F_R(t,B\xi,B\eta)$. If $g_l(\xi,\eta)$ is a coefficient in the
resulting expansion, set
$f_{R,l}(u,v)=g_l(B^{-1}u,B^{-1}v)$. An invertible linear substitution
in the polynomial vector preserves its PSD matrix representation.

Conversely, choosing any PSD coefficient polynomials in this finite
sum gives a polynomial whose restriction belongs to
$\operatorname{PD}(S^{n-1},R)$. For each $l$, its evaluation matrix
is the entrywise product of the coefficient evaluation matrix and
the Musin polynomial evaluation matrix. Both are PSD, so their entrywise product is PSD by the Schur
product property, and summing over $l$ preserves PSD. Since the coefficients do not depend on $t$
and $P_l^{n,m}(G_R;u,v,t)$ has degree $l$ in $t$, the sum has degree
at most $d$ in $t$.

For each pair $(R,l)$, write the coefficient polynomial as
\[
 f_{R,l}(u,v)=z_{R,l}(u)^TY_{R,l}z_{R,l}(v),
 \qquad Y_{R,l}\succeq0.
\]
Here $z_{R,l}$ is a finite column vector of real polynomials in $m$
variables, and $Y_{R,l}$ is a constant real symmetric matrix whose
order equals the length of this vector. Thus the construction has
the matrix form
\[
 F_R(t,u,v)=\sum_{l=0}^d
 \bigl(z_{R,l}(u)^TY_{R,l}z_{R,l}(v)\bigr)
 P_l^{n,m}(G_R;u,v,t),
 \qquad Y_{R,l}\succeq0.
\]
Once the polynomial vectors are fixed, these matrices determine the
coefficient functions and hence $F_R$. However, each coefficient may
require its own polynomial vector, and we have not imposed a common
degree bound on those vectors. Fixing $d$ alone therefore does not
specify the orders of the matrices $Y_{R,l}$. Prescribing particular
finite vectors $z_{R,l}$ can impose a further restriction on the
coefficient functions.

De Laat et al.~\cite[Theorem~4.1]{dL} make a particular choice of
these polynomial vectors. Let $z_q(\xi)$ contain all monomials of total
degree at most $q$ in the $m$ coordinates of $\xi$, including $1$,
in a fixed order. To compare at the same Gegenbauer degrees
$0,\ldots,d$, take a separate integer $\delta\ge d$ and specialize to
\[
 z_{R,l}(u)=z_{\delta-l}(B^{-1}u).
\]
With this choice, the coefficient representation becomes
\[
 f_{R,l}(u,v)=z_{\delta-l}(B^{-1}u)^T Y_{R,l}
 z_{\delta-l}(B^{-1}v),\qquad Y_{R,l}\succeq0.
\]
Here $Y_{R,l}$ is chosen relative to this prescribed monomial vector;
its order is the length of $z_{\delta-l}$. The coefficient has degree
at most $\delta-l$ in each of $u,v$ separately, and joint total degree
at most $2(\delta-l)$. Thus $d$ limits only the Gegenbauer degrees,
whereas $\delta$ also limits the coefficient degrees. The comparison
uses the same terms $0,\ldots,d$; including terms with $l>d$ changes
the search space. Their result for continuous functions gives uniform
approximation as the polynomial model grows, not an exact description
by a fixed model.

Section~\ref{sec:framework} will instead use matrices whose entries are
the values of $f_{R,l}$ on the finitely many feasible labels. These are
evaluation matrices, not the coefficient matrices $Y_{R,l}$ relative
to a polynomial vector. PSD interpolation will extend the finite tables
to coefficient polynomials, without prescribing their degrees in
advance. This retains the Gegenbauer degree bound $d$; Appendix~\ref{app:degree}
gives explicit coefficient degree bounds and the comparison with
polynomial-basis representations.

\subsection{Symmetry reduction}\label{sec:reference-order}
We have already used orthogonal symmetry to choose one function for
each reference type. The remaining symmetry is invariance under
permutations of the reference points that preserve $G_R$.
The formulas above use an ordered list $r_1,\ldots,r_m$, whereas the
original reference configuration $Q$ is an unordered set. We must
obtain the same answer from every ordering that gives the chosen
Gram matrix. We impose this invariance by averaging. Let
\[
 S_R=\{\sigma:P_\sigma^TG_RP_\sigma=G_R\},
\]
where $\sigma$ ranges over permutations of $\{1,\ldots,m\}$ and
$P_\sigma$ is the corresponding permutation matrix, defined by
$P_\sigma e_i=e_{\sigma(i)}$ for the standard coordinate vectors $e_i$.
For $m=0$, use the group with one element and the empty permutation
matrix, so the average below is the identity operation.
These are precisely the reorderings that preserve all reference
inner products. Each is realized by an orthogonal transformation
$U_\sigma$ with $U_\sigma A_R=A_RP_\sigma$: it preserves inner
products on the reference span, and can be extended by the identity
on the perpendicular space. Replace $F_R$ by the average
\begin{equation}\label{eq:coordinate-average}
 \widetilde F_R(t,u,v)=\frac1{|S_R|}
       \sum_{\sigma\in S_R}F_R(t,P_\sigma^Tu,P_\sigma^Tv).
\end{equation}
Indeed, $A_R^TU_\sigma^Tx=P_\sigma^Tu$, so each summand comes from
applying the original PSD function to the transformed points
$U_\sigma^Tx,U_\sigma^Ty$. Its matrix is PSD, and averaging preserves
that property. It also preserves polynomiality and the degree bound
in $t$, since only the label coordinates are permuted.
The average is unchanged by these reorderings.
Every orthogonal transformation carrying $R$ to itself permutes its
points in this way, possibly followed by a map fixing every reference
point. The latter already leaves $(t,u,v)$ unchanged. Thus the
average has exactly the symmetry required for~\eqref{eq:local-assembly}.

This is the operation in~\cite[Corollary~4.2]{dL}. In their coordinates,
\[
 L(A_RP_\sigma)x=B^{-1}P_\sigma^Tu
                   =L(A_R)U_\sigma^Tx.
\]
The factor $B$ is the same because
$(A_RP_\sigma)^T(A_RP_\sigma)=G_R$.
Finally choose $\gamma_Q\in O(n)$ such that $\gamma_Q(Q)=R_Q$,
and order $Q$ so that
$A_Q=\gamma_Q^TA_{R_Q}$. Substitution in~\eqref{eq:local-assembly} gives
\begin{equation}\label{eq:T-coordinate-bridge}
 T(x,y,Q)=\widetilde F_{R_Q}
              (x\cdot y,A_Q^Tx,A_Q^Ty).
\end{equation}
The resulting $T$ is continuous, PSD in its sphere arguments, and
independent of the allowed choices of ordering and orthogonal maps.
Its coefficients must still satisfy all inequalities
in~\eqref{eq:delaat4}. On code points, the labels lie in the finite set
$(\Lambda\cup\{1\})^{|Q|}$. Section~\ref{sec:framework} replaces the coefficient functions by PSD
tables on these labels, implements~\eqref{eq:coordinate-average} at the
matrix level, and removes the remaining redundant coordinates. The
result is the SDP used below for the fixed degrees $0,\ldots,d$; every
feasible solution certifies a bound by Proposition~\ref{prop:certificate}.

\section{Reductions and the final SDP formulation}\label{sec:framework}
The previous section leaves the coefficient polynomials $f_{R,l}(u,v)$
in~\eqref{eq:musin-expansion} to be chosen. The code constraints do not
use these polynomials everywhere: they use only their values on the
feasible inner-product labels. We exploit this in three steps. First we
replace each coefficient polynomial by its table of values on those
labels, without imposing any additional degree bound in $u,v$. We then
remove the redundancy coming from reference orderings and, in degree
zero, combine the reference coordinates that the union inequalities use
only through their sums. The first step gives an exact description of
the coefficient freedom for the retained Gegenbauer degrees; the next
two preserve the feasible objective values. We then assemble the SDP
used in the computation. The label-indexed viewpoint uses the same
evaluation principle that underlies the alternative constraints
of Kao--Yu~\cite[Theorem~2.13]{KY}.

\subsection{Reduction to coefficient values}
Keep an ordered, linearly independent reference set $R$ of size
$m\le n-2$, and write $G=G_R$ when no confusion is possible. Put
$s=|\Lambda|$ and $\Delta=\Lambda\cup\{1\}$. The feasible labels are
\begin{equation}\label{eq:framework-labels}
\begin{split}
 U_R^{\mathrm{gen}}&=\{u\in\Lambda^m:1-u^TG^{-1}u\ge0\},\\
 U_R&=U_R^{\mathrm{gen}}\cup\{Ge_1,\ldots,Ge_m\}.
\end{split}
\end{equation}
Here ``gen'' means that the point is outside the reference set. Indeed,
if $x\notin R$ belongs to a spherical $\Lambda$-code containing $R$,
then $A_R^Tx\in\Lambda^m$. If $x=r_i$, its label is $Ge_i$, where $e_i$ is the $i$th standard
basis vector of $\R^m$.
Conversely, Lemma~\ref{lem:framework-gram} shows that every label in
$U_R^{\mathrm{gen}}$ is realized by a unit vector with the prescribed
inner products with $R$. A feasible label in $\Delta^m$ having a
coordinate $1$ must be a reference label: $x\cdot r_i=1$ for unit
vectors implies $x=r_i$. Thus the displayed list is complete, and
\[
 |U_R|\le s^m+m,\qquad |U_R^{\mathrm{gen}}|\le s^m.
\]
The two parts of $U_R$ are disjoint. For $m=0$, $U_R$ and $U_R^{\mathrm{gen}}$ each contain
one empty label. The joint feasible triples form the finite set
\begin{equation}\label{eq:finite-domain}
 \Omega_R=\{(t,u,v):t\in\Delta,\ u,v\in U_R,\ H_{u,v,t}\succeq0\},
\end{equation}
where $H_{u,v,t}$ is the enlarged Gram matrix in
Lemma~\ref{lem:framework-gram}. By that lemma, $\Omega_R$ is the
part of the full sphere-point domain used in the code constraints:
it retains only the inner products in $\Delta$ and labels in $U_R$.
Feasibility of $u$ and $v$ separately does not make every choice of
$t\in\Delta$ possible. The PSD property is still required for arbitrary
sphere points, so specifying values on $\Omega_R$ will require a
polynomial extension with that property.

For example, take $\Lambda=\{a,b\}$ and two reference points with
$r_1\cdot r_2=\rho$. A point outside $R$ has one of the four labels
$(a,a),(a,b),(b,a),(b,b)$, provided it is feasible. The reference
points have labels $(1,\rho)$ and $(\rho,1)$. When all four outside
labels are feasible, as in our applications, the degree-zero
coefficient values form a $6\times6$ table. Its $(u,v)$ entry is
$f_{R,0}(u,v)$, not the inner product of two code points.

Evaluating a PSD polynomial at a finite set of labels always gives
a PSD table. For an SDP variable we also need the converse: every
PSD table must extend to a PSD polynomial, so that it defines a
function on all sphere points, not just the selected labels.
The next lemma constructs such an extension. Let
$W=\{w^{(1)},\ldots,w^{(N)}\}$ be a finite set of distinct labels.

\begin{lemma}\label{lem:framework-evaluation}
Every real symmetric $N\times N$ matrix $M$ has a symmetric polynomial
interpolant $f$ satisfying
\[
 f(w^{(i)},w^{(j)})=M_{ij}\qquad(1\le i,j\le N).
\]
Such an interpolant can be chosen to be a PSD polynomial if and only
if $M\succeq0$.
\end{lemma}
\begin{proof}
For each $w\in W$, choose a polynomial $\phi_w$ whose value on $W$
is $1$ at $w$ and $0$ at every other point. For example,
\[
 \phi_w(u)=\prod_{a\in W\setminus\{w\}}
               \frac{\|u-a\|^2}{\|w-a\|^2}
\]
has these values, and no denominator is zero. Let $\phi$ be the
column vector of these polynomials in the chosen order of $W$.
Then $f(u,v)=\phi(u)^TM\phi(v)$ has the prescribed evaluations.
If $M\succeq0$, this formula is also a PSD polynomial representation.

Conversely, suppose a PSD polynomial $f(u,v)=z(u)^THz(v)$,
$H\succeq0$, interpolates $M$. Let $Z$ have columns
$z(w^{(1)}),\ldots,z(w^{(N)})$. Then $M=Z^THZ\succeq0$,
since $c^TMc=(Zc)^TH(Zc)\ge0$ for every $c$.
\end{proof}

The construction chooses a PSD interpolant; an arbitrary polynomial
with the same finite values need not be PSD. For an empty label set,
take the empty matrix and the zero polynomial.

The interpolation lemma describes all coefficient tables that can occur.
For the Musin expansion we can reduce the positive-degree tables a little
further. A reference point lies entirely in the reference span, so its
perpendicular component is zero:
$1-(Ge_i)^TG^{-1}(Ge_i)=0$. The formulas of the previous section then
give $P_l^{n,m}(G;u,v,t)=0$ for $l>0$ whenever either selected point
is a reference point. Only $P_0=1$ uses reference labels.
We therefore keep all labels in $M_{R,0}$, and only labels of points
outside $R$ in $M_{R,l}$ for $l>0$. In the two-reference example above,
these matrix orders are $6$ and $4$, respectively.

\begin{samepage}
\begin{theorem}\label{thm:framework-reduction}
Fix an integer $d\ge0$. The restrictions to $\Omega_R$ of polynomial
functions $F\in\operatorname{PD}(S^{n-1},R)$ of degree at most $d$ in
$t$ are exactly the finite data
\begin{equation}\label{eq:framework-matrix-kernel}
 F(t,u,v)=\sum_{l=0}^d P_l^{n,m}(G_R;u,v,t)(M_{R,l})_{u,v},
\end{equation}
where $M_{R,0}\succeq0$ is indexed by $U_R$, and $M_{R,l}\succeq0$
is indexed by $U_R^{\mathrm{gen}}$ for $l>0$.
A positive-degree summand is zero if either label is a reference label.
No degree bound is imposed on the coefficient polynomials in $u,v$.
If $U_R^{\mathrm{gen}}$ is empty, its matrices have order zero and
all positive-degree terms vanish.
\end{theorem}
\end{samepage}
\begin{proof}
Suppose first that $F$ is such a polynomial in the PD class. Musin's
polynomial theorem gives
\[
 F(t,u,v)=\sum_{l=0}^d f_l(u,v)P_l^{n,m}(G_R;u,v,t),
 \qquad f_l\text{ a PSD polynomial}.
\]
For $l=0$, evaluate $f_0$ on all of $U_R\times U_R$ to obtain
$M_{R,0}$. For $l>0$, evaluate $f_l$ on
$U_R^{\mathrm{gen}}\times U_R^{\mathrm{gen}}$ to obtain $M_{R,l}$.
Lemma~\ref{lem:framework-evaluation} proves that these matrices are PSD.
Deleting the reference rows in positive degrees changes no function value,
because their scalar factor vanishes. Substitution gives the displayed
formula on every feasible triple.

Conversely, start with the displayed PSD matrices. By
Lemma~\ref{lem:framework-evaluation}, choose PSD polynomials
$f_l$ interpolating their entries on the corresponding label sets.
Define a polynomial on all variables by
\[
 \widehat F(t,u,v)=\sum_{l=0}^d
       f_l(u,v)P_l^{n,m}(G_R;u,v,t).
\]
For any sphere points $x_1,\ldots,x_N$, the matrix of $f_l$ values
is PSD, as is the matrix of scalar Gegenbauer function values. Their
entrywise product is PSD by the Schur product theorem; summing over
$l$ preserves PSD. Thus $\widehat F\in\operatorname{PD}(S^{n-1},R)$.
It has degree at most $d$ in $t$ and agrees with the prescribed data:
interpolation gives agreement at labels of points outside $R$, and
positive-degree terms vanish whenever a reference label occurs.
\end{proof}

The SDP can now choose the coefficient-value matrices directly. The
reverse direction of the theorem is essential: every PSD table is the
evaluation of some PSD coefficient polynomial, so this replacement does
not introduce spurious matrix choices.

\medskip\noindent\textbf{Comparison with polynomial-basis formulations.}
Suppose instead that a coefficient is represented in a prescribed
polynomial vector $z(u)$ as $f(u,v)=z(u)^THz(v)$ with $H\succeq0$, and
let $E$ be the evaluation matrix whose rows are $z(w)^T$ on the relevant
label set $W$. Its coefficient table has the form $EHE^T$ and is therefore
one of the PSD tables allowed above. Hence, for the same retained
Gegenbauer degrees, any prescribed finite polynomial-basis formulation
is contained in the coefficient-value formulation of
Theorem~\ref{thm:framework-reduction}. Equality holds once $E$ has full
row rank. Appendix~\ref{app:degree} proves this criterion and gives
explicit degree bounds for the monomial basis used by de Laat
et al.~\cite[Theorem~4.1]{dL}. Thus the present reduction removes an additional coefficient-degree
truncation while retaining the Gegenbauer degree bound $d$. It also explains
why label-indexed constraints have the same limiting coefficient freedom
as the alternative formulation of Kao--Yu
\cite[Theorem~2.13 and Corollary~2.14]{KY} once the polynomial
evaluation matrix on the allowed labels has full row rank.

The last qualification is important. The parameter $d$ still specifies
which scalar factors $P_0,\ldots,P_d$ are available. Agreement at the
finitely many allowed values of $t$ does not justify discarding higher
Gegenbauer degrees, because ordinary interpolation need not preserve the
PSD coefficient structure. Remark~\ref{rem:scalar-degree} in
Appendix~\ref{app:degree} gives a simple example.

\subsection{Reference-symmetry reduction}
The second reduction removes the dependence on the chosen ordering of
the reference points. We implement the averaging in
Section~\ref{sec:reference-order} directly on the coefficient-value
matrices. A label lists inner products in a chosen reference
order; a different allowed order permutes its coordinates.
Recall the notation
\[
 S_R=\{\sigma:P_\sigma^TG_RP_\sigma=G_R\}.
\]
Here $P_\sigma$ permutes the $m$ reference coordinates. It also permutes
each label set by $u\mapsto P_\sigma^Tu$. Define the corresponding
label permutation matrix $\Pi_\sigma$ by
$(\Pi_\sigma M\Pi_\sigma^T)_{u,v}=M_{P_\sigma^Tu,P_\sigma^Tv}$.
The ordering of each label set fixes this matrix unambiguously.

\begin{proposition}\label{prop:framework-symmetry}
Suppose $M_{R,l}\succeq0$ for each $l$. Then the matrix
\[
 \widetilde M_{R,l}=\frac1{|S_R|}\sum_{\sigma\in S_R}
                \Pi_\sigma M_{R,l}\Pi_\sigma^T
\]
is PSD and invariant under all label permutations from $S_R$.
Using these matrices in~\eqref{eq:framework-matrix-kernel} gives exactly
\[
 \widetilde F_R(t,u,v)=\frac1{|S_R|}\sum_{\sigma\in S_R}
                      F_R(t,P_\sigma^Tu,P_\sigma^Tv).
\]
These finite values have a polynomial extension in
$\operatorname{PD}(S^{n-1},R)$ with the same reference symmetry.
\end{proposition}
\begin{proof}
Simultaneously permuting the rows and columns of a PSD matrix preserves
positivity, so each summand, and then their average, is PSD. A further permutation
from $S_R$ merely reorders the summands, proving invariance.
Since $P_\sigma$ is orthogonal, $P_\sigma^TGP_\sigma=G$ implies
$GP_\sigma=P_\sigma G$, and hence
$P_\sigma G^{-1}P_\sigma^T=G^{-1}$. Thus the quantities
$u^TG^{-1}u$, $v^TG^{-1}v$, and $u^TG^{-1}v$ are unchanged when
$u,v$ are replaced by $P_\sigma^Tu,P_\sigma^Tv$.
The scalar Gegenbauer factors are therefore unchanged, and averaging
acts only on their coefficient matrices, giving the stated identity.
Finally, take the polynomial extension constructed in
Theorem~\ref{thm:framework-reduction} and average it by this same formula.
Each coordinate permutation is induced by an orthogonal transformation
of the reference configuration, so every summand remains in
$\operatorname{PD}(S^{n-1},R)$. The average is the required invariant extension.
\end{proof}

We may consequently use unrestricted PSD matrices and average during
coefficient assembly, as the program does. This does not change the feasible objective values, because averaging
twice gives the same matrix as averaging once:
$\widetilde{\widetilde M}_{R,l}=\widetilde M_{R,l}$.
Each dual inequality already uses the averaged matrices. Replacing
every $M_{R,l}$ by $\widetilde M_{R,l}$ therefore leaves those function
values, and the objective
$\alpha$, unchanged. Conversely, an invariant PSD matrix is already allowed among
the unrestricted variables. The two formulations thus have exactly the
same feasible objective values. This does not assert that an unaveraged
function is equal to its average.

A positive definiteness margin is preserved as well. For any
$\varepsilon>0$, if $M_{R,l}-\varepsilon I\succ0$, then
\[
 \widetilde M_{R,l}-\varepsilon I
 =\frac1{|S_R|}\sum_{\sigma\in S_R}
       \Pi_\sigma(M_{R,l}-\varepsilon I)\Pi_\sigma^T\succ0.
\]
Matrix-invariance equations are therefore optional in this model.
The computations below use the averaged coefficients and unrestricted
PSD decision matrices.

For each $Q$, let $R_Q$ be its chosen representative and put
$A_Q=\gamma_Q^TA_{R_Q}$, as in Section~\ref{sec:coordinates}.
For code points, the function entering the original dual is now explicitly
\begin{equation}\label{eq:assembled-kernel}
 T(x,y,Q)=\sum_{l=0}^d
 P_l^{n,|Q|}(G_{R_Q};A_Q^Tx,A_Q^Ty,x\cdot y)
 (\widetilde M_{R_Q,l})_{A_Q^Tx,A_Q^Ty}.
\end{equation}
Globally, $T$ is defined using the averaged polynomial extensions, rather
than matrix entries outside the finite label sets. The coordinate
construction ensures that it is continuous, positive semidefinite in
$x,y$ for each $Q$, and independent of the choice of $\gamma_Q$ and of the ordering used to identify
$Q$ with $R_Q$.

\subsection{Degree-zero coordinate reduction}\label{sec:compression}
The third reduction concerns only degree zero. The matrices above keep
a separate coordinate for each reference point, but the union inequalities
use these coordinates only through certain sums. Since $P_0=1$, these
sums can be identified directly from $\cB_kT(S)$.

Fix a nonempty reference set $R$ of size $m$ and put
$q_R=|U_R^{\mathrm{gen}}|$. Write $\rho_i=G_Re_i$ for its reference
labels and $M=M_{R,0}$. For $Q\subseteq S$ of this reference type,
choose an ordering identifying $Q$ with $R$, before averaging.
The degree-zero part of the inner sum in~\eqref{eq:B} has the
following three possible forms:
\begin{equation}\label{eq:compression-cases}
\begin{array}{c|c}
 S\setminus Q & \text{degree-zero contribution}\\[2pt]
 \hline
 \varnothing & \displaystyle\sum_{i,j=1}^m M_{\rho_i,\rho_j}\\[6pt]
 \{x\} & \displaystyle M_{u,u}+2\sum_{i=1}^m M_{\rho_i,u}\\[6pt]
 \{x,y\} & 2M_{u,v}
\end{array}
\end{equation}
In the second row $u=A_Q^Tx$; in the third row
$u=A_Q^Tx$ and $v=A_Q^Ty$. The points in the third row are distinct,
but their labels may coincide, in which case its contribution is
$2M_{u,u}$. If $|S\setminus Q|>2$, there is no contribution.
The first row sums all ordered pairs in $Q^2$. In the second row,
the allowed pairs are $(x,x)$ and the two orders of $x$ with each
point of $Q$. The third row allows only $(x,y),(y,x)$. This proves the
table and accounts for every allowed repeated point.

The table uses the reference labels only through their total sums.
We therefore keep one combined reference index $0$ and one index for
each outside label, with entries
\[
 \widehat M_{0,0}=\sum_{i,j=1}^m M_{\rho_i,\rho_j},\qquad
 \widehat M_{0,u}=\sum_{i=1}^m M_{\rho_i,u},\qquad
 \widehat M_{u,v}=M_{u,v}\quad(u,v\in U_R^{\mathrm{gen}}).
\]
The three rows of~\eqref{eq:compression-cases} then become,
respectively,
\[
 \widehat M_{0,0},\qquad
 \widehat M_{u,u}+2\widehat M_{0,u},\qquad
 2\widehat M_{u,v}.
\]
To form these sums by matrix multiplication, order the reference
labels first and define
\begin{equation}\label{eq:compression-maps}
 C_R=\begin{pmatrix}\mathbf1_m&0\\0&I_{q_R}\end{pmatrix},\qquad
 D_R=C_R^TC_R=\operatorname{diag}(m,I_{q_R}),\qquad
 J_R=C_RD_R^{-1}.
\end{equation}
Here $C_R,J_R\in\R^{(m+q_R)\times(q_R+1)}$ and
$D_R\in\R^{(q_R+1)\times(q_R+1)}$.
The column $\mathbf1_m$ consists of $m$ ones, so the product
$C_R^TMC_R$ forms the sums above. The matrix $J_R$ divides that reference
column by $m$, giving $C_R^TJ_R=J_R^TC_R=I_{q_R+1}$.
This provides a way to recover a full PSD matrix from the smaller one.
All entries of these maps are rational. The formulas also apply when
$q_R=0$. For $m=0$, retain the original scalar block.

\begin{proposition}\label{prop:compression}
For each nonempty $R$, replace $M_{R,0}$ by
\[
 \widehat M_{R,0}=C_R^TM_{R,0}C_R.
\]
Every PSD choice of the original blocks gives PSD compressed blocks.
Conversely, every PSD choice of the compressed blocks has the PSD lift
\begin{equation}\label{eq:compression-lift}
 M_{R,0}=J_R\widehat M_{R,0}J_R^T.
\end{equation}
Both operations preserve every value $\cB_kT(S)$ after reference
averaging, with all positive-degree blocks unchanged. Consequently,
the degree-zero block may have order $q_R+1\le s^m+1$, with exactly
the same feasible objective values in the reduced SDP formulation.
\end{proposition}
\begin{proof}
By multiplication with $C_R$, the entries of $\widehat M=C_R^TMC_R$
are precisely the sums used in~\eqref{eq:compression-cases}.
The preceding count therefore proves that each contribution,
and then their sum over $Q$, is preserved. This argument applies to
each reference ordering separately, so it also applies to their
average. More explicitly, reference permutations leave the column
$\mathbf1_m$ fixed and only permute coordinates for labels outside $R$. If
$\widehat\Pi_\sigma$ denotes this induced permutation, then
$\Pi_\sigma C_R=C_R\widehat\Pi_\sigma$ and
$\Pi_\sigma J_R=J_R\widehat\Pi_\sigma$. Thus compression and lifting
commute with reference averaging.

Both compression and lifting preserve positivity: a quadratic form
in either resulting matrix is a quadratic form in its original PSD matrix.
This proves positivity in both directions, and $C_R^TJ_R=I$ shows that
compressing the proposed lift returns $\widehat M$.
Theorem~\ref{thm:framework-reduction} extends the
lifted matrices to a function in $\operatorname{PD}(S^{n-1},R)$.
The lift therefore defines an admissible dual function, and the equalities
above preserve all its required union values.
\end{proof}

This proves that compression preserves the assembled inequalities
and their feasible objective values. The individual values $T(x,y,Q)$
may change when the reference coordinates are combined. They need not
be recovered separately to prove the bound. We next address the
additional strict matrix margins used for numerical verification.

For numerical verification we also want a strictly positive matrix
margin. Compressing a full margin $\varepsilon I$ gives
$C_R^T(\varepsilon I)C_R=\varepsilon D_R$, which explains why the
smaller matrix must be tested against $\varepsilon D_R$.

The basic lift $J_R\widehat M_{R,0}J_R^T$ has zero quadratic form on
vectors supported on the reference coordinates whose sum is zero.
To make the full matrix strictly positive, we add a positive term
on precisely these unused directions. Put
\begin{equation}\label{eq:compression-projector}
 P_R^{\perp}=I-C_RD_R^{-1}C_R^T.
\end{equation}
This matrix subtracts the mean from the reference coordinates and
sets the outside coordinates to zero. It is the orthogonal projector
onto $\ker C_R^T$: it is symmetric, its square equals itself, and
$C_R^TP_R^{\perp}=0$. Adding a multiple of it changes none of the
three sums in~\eqref{eq:compression-cases}.

\begin{lemma}\label{lem:compression-margin}
Let $\varepsilon>0$. If $M_{R,0}-\varepsilon I\succ0$, then
$\widehat M_{R,0}-\varepsilon D_R\succ0$.
Conversely, if $\widehat M_{R,0}-\varepsilon D_R\succ0$, then, for
any $\tau>0$, the lift
\begin{equation}\label{eq:compression-strict-lift}
 M_{R,0}=J_R\widehat M_{R,0}J_R^T+
                 (\varepsilon+\tau)P_R^{\perp}
\end{equation}
satisfies $M_{R,0}-\varepsilon I\succ0$ and gives the same union values.
\end{lemma}
\begin{proof}
The first assertion follows from
$\widehat M_{R,0}-\varepsilon D_R
 =C_R^T(M_{R,0}-\varepsilon I)C_R$ and the full column rank of $C_R$.
For the converse, $J_RD_RJ_R^T=I-P_R^{\perp}$ gives
\[
 M_{R,0}-\varepsilon I
 =J_R(\widehat M_{R,0}-\varepsilon D_R)J_R^T+\tau P_R^{\perp}.
\]
Every vector has a unique decomposition $z=C_Rw+r$ with
$C_R^Tr=0$. Its quadratic form in this expression is
$w^T(\widehat M_{R,0}-\varepsilon D_R)w+\tau\|r\|^2$,
which is positive for $z\ne0$. Finally,
$C_R^TP_R^{\perp}C_R=0$, so the added term changes none of the
three contributions in~\eqref{eq:compression-cases}.
\end{proof}

\subsection{The final SDP}\label{sec:finite-sdp}
After the three reductions, the optimization problem can be written
directly in terms of the resulting matrix variables. The notation used
in the preceding steps is summarized below.
\begin{center}
\begin{tabular}{@{}ll@{}}
\toprule
Notation & Role\\
\midrule
$M_{R,l}$ & Coefficient matrix before compression\\
$\widetilde M_{R,l}$ & Average over reference permutations\\
$\widehat M_{R,0}$ & Compressed degree-zero matrix\\
$H_{R,l}$ & Matrix variable in the final SDP\\
\bottomrule
\end{tabular}
\end{center}
For $R=\varnothing$, set $\widehat M_{R,0}=M_{R,0}$; no
compression is needed for this scalar block.
The unknowns are a real number $\alpha$ and the following symmetric
matrices, one for each reference type $R\in\mathcal R_{k-2}$:
\[
 H_{R,0}=\widehat M_{R,0},\qquad
 H_{R,l}=M_{R,l}\quad(1\le l\le d).
\]
For $m=|R|\ge1$, their index sets are
$\{0\}\cup U_R^{\mathrm{gen}}$ in degree zero and
$U_R^{\mathrm{gen}}$ in positive degrees. Their orders are therefore
$q_R+1$ and $q_R$, where $q_R=|U_R^{\mathrm{gen}}|$.
Here $0$ denotes the combined reference index. It is separate from
every label in $U_R^{\mathrm{gen}}$, including a label whose coordinates
are all zero.
When $R=\varnothing$, every degree has one scalar variable.
For two inner products, the orders are at most $2^m+1$ and $2^m$.
The parameters $n,\Lambda,k,d$, all Gram types and all coefficient
matrices defined below are fixed before optimization begins.

Here is an explicit way to form the coefficient matrices. On any one
of the index sets above, let $e_u$ be its coordinate vector and put
\[
 E_{u,v}=\tfrac12(e_ue_v^T+e_ve_u^T),\qquad
 \langle A,H\rangle=\operatorname{tr}(AH).
\]
Then $\langle E_{u,v},H\rangle=H_{u,v}$, including when $u=v$.
Thus an off-diagonal entry of $E_{u,v}$ is $1/2$, whereas
$E_{u,u}$ has a single diagonal entry $1$.

Take a realizable Gram type $G_S=(t_{ij})_{i,j=1}^N$, with
$1\le N\le k$, and use $S=\{1,\ldots,N\}$ as its point indices.
For each reference type $R$ of size $m$, let $\mathcal I_R(S)$ be
the ordered tuples $q=(q_1,\ldots,q_m)$ of distinct indices
such that
\[
 (t_{q_iq_j})_{i,j=1}^m=G_R.
\]
Recall that $S_R$ is the group of reference permutations preserving
$G_R$. Each reference subset of type $R$ occurs exactly $|S_R|$
times in this list. For one such tuple put
$Q=\{q_1,\ldots,q_m\}$ and
$u_x=(t_{q_1x},\ldots,t_{q_mx})^T$ for $x\in S\setminus Q$.
For $m\ge1$, define matrices $Z_{S,q,l}$ by the following table;
each $E_{u,v}$ uses the index set of the indicated degree:
\[
\begin{array}{c|c|c}
 S\setminus Q & Z_{S,q,0} & Z_{S,q,l}\ (l>0)\\[2pt]
 \hline
 \varnothing & E_{0,0} & 0\\[3pt]
 \{x\} & E_{u_x,u_x}+2E_{0,u_x}
   & P_l^{n,m}(G_R;u_x,u_x,1)E_{u_x,u_x}\\[3pt]
 \{x,y\} & 2E_{u_x,u_y}
   & 2P_l^{n,m}(G_R;u_x,u_y,t_{xy})E_{u_x,u_y}\\[3pt]
 |S\setminus Q|>2 & 0 & 0
\end{array}
\]
In the third row, $x$ and $y$ are distinct, but their labels may be
equal; the coefficient remains $2$ in that case. The degree-zero
column is the compression table~\eqref{eq:compression-cases} written
as matrix inner products. For positive degrees, every term involving
a reference point vanishes. The second row then retains only
$(x,x)$, and the third retains the two orders of $(x,y)$.

For $R=\varnothing$, the list $\mathcal I_R(S)$ contains the empty
tuple once and $|S_R|=1$. All coefficient matrices are scalars:
\[
 Z_{S,(),l}=\begin{cases}
 1,&N=1,\\
 2P_l^n(t_{12}),&N=2,\\
 0,&N\ge3,
 \end{cases}
 \qquad 0\le l\le d.
\]
The fixed coefficient matrices for the union inequality are now
\[
 A_{S,R,l}=\frac1{|S_R|}
       \sum_{q\in\mathcal I_R(S)}Z_{S,q,l}.
\]
If $m>N$, the sum is empty and its value is the zero matrix.
Summing over all reference orderings and dividing by $|S_R|$ gives
exactly the reference average for each subset $Q$; it introduces no
additional multiplicity. Thus reference averaging is built into the
fixed coefficient matrices $A_{S,R,l}$. The unknown matrices $H_{R,l}$
need not satisfy invariance equations.

Fix $n$, a finite $\Lambda\subset[-1,1)$, $2\le k\le n$ and $d\ge0$,
under the reference-independence assumption above. The final SDP is
\begin{equation}\label{eq:dual}
\begin{aligned}
 \text{minimize}\quad&\alpha,\\
 \text{subject to}\quad
 &\sum_{R\in\mathcal R_{k-2}}\sum_{l=0}^d
       \langle A_{S,R,l},H_{R,l}\rangle
       \le\begin{cases}
              \alpha-1,&|S|=1,\\
              -2,&|S|=2,\\
              0,&3\le|S|\le k,
            \end{cases}\\
 &H_{R,l}\succeq0
       \quad(R\in\mathcal R_{k-2},\ 0\le l\le d),\\
 &\alpha\ge0.
\end{aligned}
\end{equation}
There is one linear inequality for every realizable Gram type $S$,
up to simultaneous row and column permutation. Only $\alpha$ and
the entries of the matrices $H_{R,l}$ are optimized.

To connect this matrix problem with the earlier function bound,
lift $H_{R,0}$ by~\eqref{eq:compression-lift}, interpolate the
coefficient matrices, and take the reference average. The table
above and~\eqref{eq:assembled-kernel} then give
\[
 \sum_{R\in\mathcal R_{k-2}}\sum_{l=0}^d
       \langle A_{S,R,l},H_{R,l}\rangle=\cB_kT(S).
\]
Consequently, every feasible solution of~\eqref{eq:dual} gives
$|C|\le\alpha$ by Proposition~\ref{prop:certificate}.
This also shows that the explicit matrix problem has the same feasible
objective values as the preceding compressed formulation.
We keep the singleton inequality rather than replacing it by an equality;
strictly feasible certificates need not attain the minimum.
It already implies $\alpha\ge1$, so $\alpha\ge0$ is redundant.

For the twelve applications in Section~\ref{sec:application},
Table~\ref{tab:blocks} gives the matrix orders before and after
compression. At six points, reference sets have at most four points,
so the largest degree-zero block decreases from order $20$ to $17$.

\begin{table}[ht]
\centering\small
\begin{tabular}{@{}rrrrr@{}}
\toprule
$m$ & Reference types & Full degree zero & Compressed degree zero & Positive degrees\\
\midrule
0 & 1 & 1 & 1 & 1\\
1 & 1 & 3 & 3 & 2\\
2 & 2 & 6 & 5 & 4\\
3 & 4 & 11 & 9 & 8\\
4 & 11 & 20 & 17 & 16\\
\bottomrule
\end{tabular}
\caption{Coefficient matrix orders for the twelve six-point problems.
The full degree-zero column is before the degree-zero coordinate
reduction; the compressed column is the final degree-zero block. There
is one block at each degree $0,\ldots,5$ for each reference type.}
\label{tab:blocks}
\end{table}

\section{Applications to strongly regular graphs}\label{sec:application}
We now apply the SDP of Section~\ref{sec:framework} to strongly regular
graphs. The spherical embedding in the next subsection is general. We then
specialize to the twelve parameter sets in Koolen--Gebremichel~\cite[Table~1]{KG},
check that the six-point models used for those cases are complete and
well defined, and finally state the certified bounds.

\subsection{From graph parameters to the spherical problem}\label{sec:spherical-reduction}
Let a primitive strongly regular graph have parameters
$(v,k,\lambda,\mu)$; here $k$ denotes valency. We give the spectral
calculation, following the spherical embedding described in
Brouwer and Van Maldeghem~\cite{BVM}.
If $A$ is the adjacency matrix, counting walks of length two gives
\[
A^2=(k-\mu)I+(\lambda-\mu)A+\mu J,
\]
where $J$ is the all-ones matrix. On the space perpendicular to the
all-ones vector, $J$ vanishes. The two restricted eigenvalues $r>s$
are the roots of $z^2-(\lambda-\mu)z-(k-\mu)=0$.
Their multiplicities $n,g$ satisfy $n+g=v-1$ and
$k+nr+gs=0$, the latter because $\tr A=0$.
The orthogonal projector onto the $r$-eigenspace is
\[
E=\frac{A-sI-(k-s)J/v}{r-s}.
\]
It acts as the identity on that eigenspace and as zero on the
other eigenspace and on the all-ones vector. Its diagonal is
constant and its trace is $n$, so $E_{ii}=n/v$.
Let $e_i$ be the $i$th standard basis vector of $\mathbb R^v$.
Then $p_i=\sqrt{v/n}\,Ee_i$ are unit vectors in an
$n$-dimensional space, with $p_i\cdot p_j=(v/n)E_{ij}$.

Let $c_1,c_2$ be the off-diagonal entries of $E$ on edges and
nonedges. From the diagonal of $AE=rE$ and the row sums of $EJ=0$,
\[
kc_1=rn/v,\qquad n/v+kc_1+(v-k-1)c_2=0.
\]
Consequently,
\begin{equation}\label{eq:embedding}
p_i\cdot p_j=
\begin{cases}
b=r/k,&i\sim j,\\
a=-(r+1)/(v-k-1),&i\not\sim j,\ i\ne j.
\end{cases}
\end{equation}
The complement has adjacency matrix $J-I-A$ and acts on the same
eigenspace with eigenvalue $-1-r$. Since the graph and its complement
are connected, their nonprincipal eigenvalues are strictly smaller
than their valencies. Thus $r<k$ and $-1-r<v-k-1$, so $b<1$ and
$a<1$. The vectors $p_i$ are therefore distinct, and the graph gives
a spherical two-distance set in $S^{n-1}$ with one point for each vertex.
The complement has the same projector $E$ and therefore the same
spherical embedding. Its parameters are
\[
(v,v-k-1,v-2k+\mu-2,v-2k+\lambda).
\]

For the twelve parameter sets considered below, both $\mu$ and the
complementary value $v-2k+\lambda$ are positive, so the graph and its
complement would be connected. In all twelve cases the two inner products
from~\eqref{eq:embedding} satisfy $a<b<1$, so different vertices give
distinct unit vectors. Table~\ref{tab:embedding} records the spherical
data for the four parameter sets that will be excluded.

\begin{table}[ht]
\centering\small
\begin{tabular}{@{}lrrrc@{}}
\toprule
Parameters & $r$ & $n$ & $a$ & $b$\\
\midrule
$(351,140,73,44)$ & $32$ & $26$ & $-11/70$ & $8/35$\\
$(550,162,75,36)$ & $42$ & $33$ & $-1/9$ & $7/27$\\
$(703,182,81,35)$ & $49$ & $37$ & $-5/52$ & $7/26$\\
$(1344,221,88,26)$ & $65$ & $56$ & $-1/17$ & $5/17$\\
\bottomrule
\end{tabular}
\caption{Data for the spherical embeddings required by the four
excluded parameter sets. Here $r$ has multiplicity $n$; the other
restricted eigenvalue is $-3$.}
\label{tab:embedding}
\end{table}

\paragraph{Relation with the cubic Krein condition.}
For completeness, take either restricted eigenvalue $\theta$ with
multiplicity $n$, and put $h=v-k-1$. Then
$b=\theta/k$, $a=-(\theta+1)/h$ and $1+kb+ha=0$. Since
$P_3^n(t)=((n+2)t^3-3t)/(n-1)$,
\[
1+kP_3^n(b)+hP_3^n(a)
=\frac{n+2}{n-1}
\left(1+\frac{\theta^3}{k^2}-\frac{(\theta+1)^3}{h^2}\right).
\]
The expression in parentheses is the corresponding Krein condition
written as a nonnegative quantity; see~\cite[Section~1.3.4]{BVM}.
This identifies the cubic condition for the fixed pair counts of the
embedding; it makes no assertion about higher degrees of the spherical LP.

From this point on, we specialize to the twelve parameter sets listed
in Table~\ref{tab:computed}.

\subsection{Checking the twelve SDP models}\label{sec:twelve-models}
To apply Section~\ref{sec:framework} through six points, two things must
be checked for these twelve spherical problems. First, every reference
Gram matrix of order at most four must be positive definite, so that the
inverses in the projection formulas are well defined. Second, the
feasible labels and Gram types used to form the constraints must be
completely enumerated. The first point follows from a uniform estimate;
the second is a finite exact computation.

Return to $k$ as the multipoint level. All twelve parameter sets satisfy
$\beta=\max_{t\in\Lambda}|t|\le5/17$. For $k\le6$, each reference
set has $m\le4$ points. For every $z\in\mathbb R^m$, its Gram matrix obeys
\[
z^TG_Rz\ge\sum_i z_i^2-2\beta\sum_{i<j}|z_iz_j|
\ge(1-(m-1)\beta)\|z\|^2
\ge\frac2{17}\|z\|^2,
\]
using $2|z_iz_j|\le z_i^2+z_j^2$. Thus every reference Gram matrix
is positive definite. Also $n-m\ge22$, so all projection and
Gegenbauer formulas in Sections~\ref{sec:coordinates}--\ref{sec:framework}
apply.

Exact enumeration gives the remaining completeness information.
All Gram patterns on at most five points are positive definite in each
of the twelve problems. In particular every generic label in
$\Lambda^m$ is feasible, and
\begin{equation}\label{eq:labels}
U_R=\Lambda^m\cup\{G_Re_1,\ldots,G_Re_m\},\qquad
U_R^{\mathrm{gen}}=\Lambda^m.
\end{equation}
Since $|\Lambda|=2$, we have $q_R=|U_R^{\mathrm{gen}}|=2^m$ for
$m\le4$. Hence the block orders in Table~\ref{tab:blocks} are exact
for these twelve problems. For nonempty reference sets, the orders are
$2^m+1$ in degree zero after compression and $2^m$ in positive degrees.
For the empty reference set, each block has order one.

At six points all $156$ Gram types are PSD, but one has rank five.
It consists of two triples, with inner product $a$ within each triple
and $b$ between them. All twelve parameter sets satisfy $3b=1+2a$;
the eigenvalues of this Gram matrix are $1-a$ with multiplicity four,
$2(1+2a)$, and $0$. This singular configuration is retained in every
six-point SDP. No singular matrix is inverted, because the inverses in
Sections~\ref{sec:coordinates}--\ref{sec:framework} involve only
reference Gram matrices with at most four points.

The estimate above is an analytic proof of reference independence.
The Gram-type and label enumeration is carried out in exact arithmetic;
Section~\ref{sec:workflow} explains how the model builder reproduces
these lists. This checks completeness of the SDP models. The exact
feasibility of the computed certificates is a separate check described
in the same section.

\subsection{Certified bounds and graph exclusions}
Table~\ref{tab:computed} gives the bounds obtained from~\eqref{eq:dual}
for all twelve parameter sets in Koolen--Gebremichel~\cite[Table~1]{KG}.
The dimension $n$ is the multiplicity of the positive restricted
eigenvalue, and the two inner products are given by~\eqref{eq:embedding}.
Every column uses Gegenbauer degrees $0,\ldots,5$ ($d=5$).
The table reports these degree-five bounds alone, without taking their
minimum with the classical absolute or Musin bounds. Each entry is
$\lfloor\alpha\rfloor$, where feasibility was checked in exact rational
arithmetic. These are proved upper bounds; the SDP optimum may be smaller.

\begin{table}[ht]
\centering\small
\setlength{\tabcolsep}{5pt}
\begin{tabular}{@{}lrrrrrr@{}}
\toprule
Parameters & $n$ & LP ($k=2$) & $k=3$ & $k=4$ & $k=5$ & $k=6$\\
\midrule
$(288,105,52,30)$ & $27$ & $433$ & $407$ & $389$ & $379$ & $371$\\
$(300,117,60,36)$ & $26$ & $381$ & $368$ & $358$ & $352$ & $348$\\
$(351,140,73,44)$ & $26$ & $377$ & $365$ & $357$ & $352$ & \textbf{348}\\
$(375,102,45,21)$ & $34$ & $1020$ & $759$ & $648$ & $576$ & $526$\\
$(405,132,63,33)$ & $30$ & $616$ & $530$ & $486$ & $463$ & $440$\\
$(441,88,35,13)$ & $44$ & $1579$ & $1546$ & $1123$ & $940$ & $805$\\
$(476,133,60,28)$ & $34$ & $1026$ & $750$ & $643$ & $576$ & $527$\\
$(540,147,66,30)$ & $35$ & $1062$ & $812$ & $683$ & $605$ & $548$\\
$(550,162,75,36)$ & $33$ & $906$ & $684$ & $596$ & \textbf{546} & \textbf{506}\\
$(575,112,45,16)$ & $46$ & $1740$ & $1682$ & $1247$ & $1057$ & $897$\\
$(703,182,81,35)$ & $37$ & $1147$ & $950$ & $770$ & \textbf{662} & \textbf{592}\\
$(1344,221,88,26)$ & $56$ & $2890$ & $2511$ & $1787$ & $1585$ & \textbf{1307}\\
\bottomrule
\end{tabular}
\caption{Certified cardinality bounds with $d=5$ in every column,
including the LP column. No minimum with the classical bounds is taken.
Bold entries are below the number of vertices required by the graph.}
\label{tab:computed}
\end{table}

For $k=2$, only the empty reference set occurs. Each coefficient
matrix is a nonnegative scalar $c_l$, and~\eqref{eq:dual} reduces to
\[
\min\left\{1+\sum_{l=0}^5c_l:
c_l\ge0,\quad\sum_{l=0}^5c_lP_l^n(t)\le-1\quad(t\in\Lambda)\right\}.
\]
The factor $-1$ comes from dividing the pair inequality by two, since
selected pairs are ordered. Rational primal and dual feasible solutions
with equal objective values certify each degree-five LP optimum; their
integer parts agree with the displayed LP column. The bounds decrease
along every row. At five points they exclude the $550$- and
$703$-vertex parameters; at six points they also exclude the $351$-
and $1344$-vertex parameters.

The four six-point certificates give the following strict comparisons:
\[
\begin{array}{c|rrrr}
v&351&550&703&1344\\ \hline
\alpha<&348.872969&506.100183&592.819146&1307.387230.
\end{array}
\]
Each terminating decimal is an exact rational number chosen above the
saved value of $\alpha$.

\begin{theorem}\label{thm:spherical}
For the pairs $(n,\{a,b\})$ in Table~\ref{tab:embedding}, every
spherical $\{a,b\}$-code has at most $348$, $506$, $592$ and
$1307$ points, respectively.
\end{theorem}
\begin{proof}[Proof of Theorems~\ref{thm:spherical} and~\ref{thm:main}]
Each of the four rational certificates passes
Proposition~\ref{prop:exact} with $k=6,d=5$, using the checks
described in Section~\ref{sec:workflow}. Proposition~\ref{prop:certificate}
bounds the code size by its saved value of $\alpha$. The displayed
comparisons and integrality give $348$, $506$, $592$ and $1307$.
The spherical embeddings would require $351$, $550$, $703$ and $1344$
distinct points. Each exceeds the corresponding bound. The complement
relation proves the complementary exclusions.
\end{proof}

Separate five-point certificates using only degrees $0,1,2$ already
give the bounds $546$ and $662$ for the $550$- and $703$-vertex rows.
Table~\ref{tab:computed} uses $d=5$ throughout so that all columns use
the same degrees. Only $d=5$ is reported for the six-point computations.

\section{Computation and exact verification}\label{sec:workflow}\label{sec:cert}\label{sec:rigorous}
The SDP solver is used only to find candidate matrices. A bound is
proved after a specified rational candidate passes exact tests for the
original rational SDP. We first define the certificate and its acceptance
test, then describe the exact model construction, the high-precision
search, and independent cross-checks.

\subsection{The certificate and its exact verification}\label{sec:exact-certificate}
For a fixed input $(n,\Lambda,k,d)$, a certificate consists of a rational
number $\alpha$ and one symmetric rational matrix for each block
in~\eqref{eq:dual}. In these applications all Gram matrices, labels and
coefficient matrices are rational. The solver's status and approximate
slack variables are not used to establish feasibility.

Number the coefficient blocks by $b$, let $H_b$ denote the matrix in
block $b$, and let $q_b$ be its order. For a nonempty reference set and
degree zero, put $H_b=\widehat M_{R,0}$ and $W_b=D_R$, where
$D_R=\operatorname{diag}(m,I_{q_R})$. For every other block, put
$H_b=M_{R,l}$ and $W_b=I$. Thus $W_b\succ0$ in every case.
Move the right-hand side of each inequality in~\eqref{eq:dual} to the
left and write the exact residual as
\begin{equation}\label{eq:residual}
r_S=c_S-e_S\alpha+
\sum_b\sum_{i\le j}a_{S,b,i,j}(H_b)_{ij},
\end{equation}
where $c_S=1,2,0$ for $|S|=1,2,\ge3$, respectively, and $e_S=1$ only
for singletons and $e_S=0$ otherwise. If block $b$ corresponds to $(R,l)$, then
\[
a_{S,b,i,i}=(A_{S,R,l})_{ii},\qquad
a_{S,b,i,j}=2(A_{S,R,l})_{ij}\quad(i<j).
\]
The factor two combines the symmetric off-diagonal entries in the
matrix inner product. All coefficients in~\eqref{eq:residual} are exact
rational numbers.

The SDP itself only requires PSD blocks and nonpositive residuals. For
verification we use the stronger sufficient conditions below: a positive
definiteness margin for every block and a strictly negative residual for
every inequality.

\begin{proposition}\label{prop:exact}
Fix one of the spherical problems in Table~\ref{tab:computed}, a level
$2\le k\le6$ and an integer $d\ge0$. Suppose a rational candidate has
the complete block and label structure of~\eqref{eq:dual}, $\alpha\ge0$
and symmetric matrices $H_b$. Let $\varepsilon,\eta>0$ be rational. If
\[
H_b-\varepsilon W_b
=L_b\operatorname{diag}(d_{b,1},\ldots,d_{b,q_b})L_b^T,
\qquad d_{b,i}>0,
\]
for every block, with $L_b$ unit lower triangular, and if
$r_S<-\eta$ for every realizable Gram type on $1$--$k$ points, then the
candidate is strictly feasible for~\eqref{eq:dual}. Every spherical
$\Lambda$-code in $\mathbb R^n$ has at most $\lfloor\alpha\rfloor$
points.
\end{proposition}
\begin{proof}
The factorization proves $H_b-\varepsilon W_b\succ0$. Since
$W_b\succ0$, all coefficient blocks are positive definite.
Lemma~\ref{lem:compression-margin} supplies full-label lifts with the
same residuals and a positive definiteness margin. The residual tests
cover every required inequality because the Gram-type list is complete
and the union sum is invariant under relabelling.
Proposition~\ref{prop:certificate} gives $|C|\le\alpha$, and integrality
gives the stated bound. The singleton inequality forces $\alpha\ge1$.
\end{proof}

The verifier performs the shifted-matrix tests by exact rational
$LDL^T$ factorization; positivity of every pivot certifies positive
definiteness. All arithmetic operations and sign comparisons are exact.
All sixty entries in Table~\ref{tab:computed} pass the test with
\begin{equation}\label{eq:margin}
H_b-10^{-30}W_b\succ0,\qquad r_S<-10^{-26}.
\end{equation}
The two additional degree-two certificates satisfy the same margins.

For a candidate stored in compressed form, the verifier also lifts each
nonempty degree-zero block back to the full label space. Taking
$\varepsilon=10^{-30}$ and
$P_R^\perp=I-C_RD_R^{-1}C_R^T$ as in~\eqref{eq:compression-projector},
it forms
\[
M_{R,0}=J_R\widehat M_{R,0}J_R^T+2\varepsilon P_R^\perp,
\]
checks $M_{R,0}-\varepsilon I\succ0$, and reevaluates every original
inequality. The compressed and lifted residuals must agree exactly.
Thus the margins in~\eqref{eq:margin} are proved inequalities, not
numerical tolerances.

The acceptance tests use arbitrary-size Python integers and
\texttt{fractions.Fraction}. The reported certificates were checked
with Python~3.12.14; the distributed programs support Python~3.9 and
later. No floating-point arithmetic enters these feasibility tests.
Every finite decimal imported from the solver is interpreted as the
exact rational number specified by its printed mantissa and exponent.
The integer bounds and the decimal comparisons in
Section~\ref{sec:application} are likewise checked by exact rational
arithmetic.

\subsection{Building the exact SDP}\label{sec:build-exact-sdp}
Starting from $(n,\Lambda,k,d)$, the program constructs the rational SDP
of Section~\ref{sec:finite-sdp}. It enumerates symmetric Gram patterns
with diagonal $1$ and off-diagonal entries in $\Lambda$, groups them by
simultaneous row and column permutations, and tests realizability by
exact PSD and rank computations. For the twelve problems in
Section~\ref{sec:twelve-models}, the numbers of Gram types at orders
$0,1,2,3,4,5,6$ are $1,1,2,4,11,34,156$; the order-six list covers all
$2^{15}$ labelled patterns and contains the singular rank-five type
described there.

The program then enumerates the feasible labels and reference
permutations. Formula~\eqref{eq:homogeneous-recurrence} evaluates the
Gegenbauer factors using rational Gram inverses. The union condition
in~\eqref{eq:B} is applied to each reference subset and ordered selected
pair. Reference averaging is built into the fixed coefficients, while
the decision matrices remain unrestricted PSD variables.

For a nonempty reference set $R$, let $B_{S,R,0}$ be the degree-zero
coefficient matrix before compression, indexed by all labels in $U_R$.
The corresponding compressed coefficient matrix is $A_{S,R,0}$ from
Section~\ref{sec:finite-sdp}. Using $C_R,J_R$ from
\eqref{eq:compression-maps}, the builder and verifier check
\begin{equation}\label{eq:coefficient-factorization}
A_{S,R,0}=J_R^TB_{S,R,0}J_R,\qquad
B_{S,R,0}=C_RA_{S,R,0}C_R^T.
\end{equation}
Consequently,
$\langle B_{S,R,0},M_{R,0}\rangle
=\langle A_{S,R,0},C_R^TM_{R,0}C_R\rangle$.
These are the coefficients used in Section~\ref{sec:finite-sdp}.

At $k=6,d=5$, the exact model has $19$ reference types, $114$
coefficient blocks and $208$ union inequalities. The largest degree-zero
block has order $17$, and the corresponding positive-degree blocks have
order $16$. One scalar block for $\alpha$ and one nonnegative slack for
each inequality give $323$ SDPA blocks. These counts are checked against
the generated input.

\begin{figure}[ht]
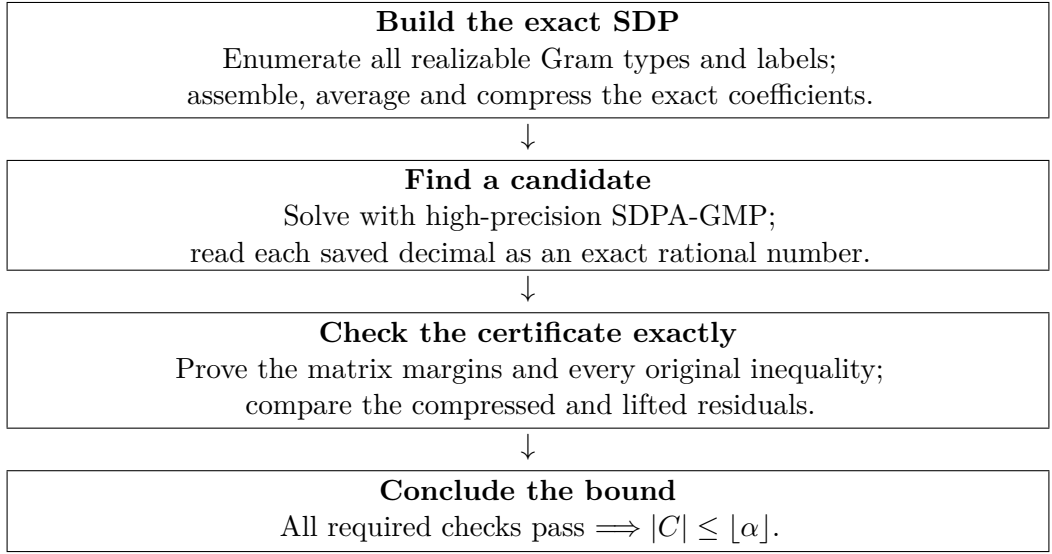

\centering
\begin{tabular}{c}
\fbox{\parbox{0.88\textwidth}{\centering
\textbf{Build the exact SDP}\par
Enumerate all realizable Gram types and labels;\par
assemble, average and compress the exact coefficients.}}\\[2pt]
$\downarrow$\\[2pt]
\fbox{\parbox{0.88\textwidth}{\centering
\textbf{Find a candidate}\par
Solve with high-precision SDPA-GMP;\par
read each saved decimal as an exact rational number.}}\\[2pt]
$\downarrow$\\[2pt]
\fbox{\parbox{0.88\textwidth}{\centering
\textbf{Check the certificate exactly}\par
Prove the matrix margins and every original inequality;\par
compare the compressed and lifted residuals.}}\\[2pt]
$\downarrow$\\[2pt]
\fbox{\parbox{0.88\textwidth}{\centering
\textbf{Conclude the bound}\par
All required checks pass $\Longrightarrow |C|\le\lfloor\alpha\rfloor$.}}\\
\end{tabular}
\caption{The numerical solver is used only to find a candidate; the
final proof is an exact rational verification of the original SDP.}
\label{fig:workflow}
\end{figure}

\subsection{Finding a candidate}
Following de Laat et al.~\cite[Section~6.2, p.~554]{dL}, we use
high-precision SDPA-GMP to find candidates and then verify the original
rational inequalities. We use the \path{sdpa-gmp-7.1.3.tar.gz}
distribution supplied by de Laat, built with GMP~6.3.0; see~\cite{SDPA}.
The requested precision is $1500$ bits, rounded up to $1536$ bits by the
recorded 64-bit GMP build. Both \texttt{epsilonStar} and
\texttt{epsilonDash} are $10^{-16}$, and the iteration limit is $2000$.
These parameters guide the search; they are not certificate acceptance
conditions.

For the direct compressed six-point search, put $\varepsilon=10^{-30}$.
The generator writes the unchanged blocks as $M=X+2\varepsilon I$ and
the nonempty degree-zero blocks as
$\widehat M=X+2\varepsilon D_R$, with $X\succeq0$. It adjusts the
constant terms and requests residuals at most $-10^{-26}$. These shifts
help the solver find candidates that pass the strict tests.

For the twelve six-point runs, the exact coefficients are saved before
conversion to the SDPA input format. The input uses $600$ significant
decimal digits per nonzero coefficient. Off-diagonal scalar
coefficients are halved to match SDPA's matrix-inner-product convention.
A separate check reads the input file, verifies its objective,
right-hand sides, block sizes, shifts and trace factors, and computes
every input rounding error as an exact rational number.

In these six-point runs, the candidate matrices and $\alpha$ are read
from \texttt{yMat}, whose
objective under this export convention is $-\alpha$. The parser checks
dimensions and exact symmetry, converts the printed finite decimals to
rational numbers, and adds the recorded shifts exactly. A second parser
compares every saved entry, including $\alpha$ and the auxiliary slacks,
with \texttt{yMat} plus those shifts; any mismatch is rejected.

No error bound for the internal iterations of SDPA-GMP enters the final
proof. Input rounding, solver error or output truncation may cause a
candidate to fail the exact tests, but none is accepted as a small
violation of the mathematical constraints.

\subsection{Independent checks}
For every six-point case, a separate calculation enumerates reference
subsets and orderings and recomputes the degree-zero coefficients from
the three union cases in~\eqref{eq:compression-cases}. It compares the
original and compressed coefficients, confirms that positive-degree
coefficients are unchanged, and checks coverage of the labelled Gram
patterns.

The main model builder and verifier share some code for reconstructing
coefficients. To reduce this dependence, an additional Python checker
reconstructed every coefficient at degrees $0,\ldots,5$ for all twelve
six-point cases without importing either program. It used the finite
Gegenbauer sum~\cite[Eq.~18.5.10]{DLMF}, enumerated all reference
orderings with the required Gram matrix, and compared every scalar
coefficient with the saved original model. It tested Gram feasibility by
all principal minors and shifted decision matrices by positive leading
principal minors, using integer arithmetic. It then reconstructed the
full-label matrices and reevaluated every original residual. The same
strict margins and integer bounds were obtained.

For each degree-five LP, rational primal and dual feasible solutions
have exactly equal objective values. Python checks their feasibility,
objective equality and complementary slackness, thereby proving the LP
optima. Their exact integer parts give the LP column of
Table~\ref{tab:computed}. We retain the inputs, rational candidates,
solver outputs, verification reports and file hashes identifying the
data and programs used for these checks.

\section*{Acknowledgments}
We thank Zilin Jiang for suggesting the use of semidefinite programming
to investigate existence questions for strongly regular graphs, a
direction that provided an important starting point for this work. We
are also grateful to him for several valuable discussions and for many
insights that helped shape the development of the project.

We acknowledge the coding assistance of OpenAI's GPT-6 in implementing
and debugging the programs, which substantially reduced the time
required for this work. The authors take responsibility for the
mathematical arguments, computations, and conclusions.

\appendix
\section{Programs and certificates}\label{app:reproduce}
The Python programs and the four six-point certificates proving
Theorem~\ref{thm:main} are available at
\begin{center}
\url{https://github.com/Wideep886/srg_sdp}.
\end{center}
This paper refers to
\href{https://github.com/Wideep886/srg_sdp/tree/de824a2bd72f45843df2f226a626b02a676e9140}{revision \texttt{de824a2}}.
The repository contains the exact models, rational candidates, solver
outputs, and verification programs for these four certificates, together
with instructions in English and Traditional Chinese. They give the
integer bounds $348$, $506$, $592$, and $1307$ reported in
Section~\ref{sec:application}.

Checking a saved certificate requires only Python~3.9 or later and its
standard library; no SDP solver is needed. The verifier recomputes the
matrix and residual tests, the degree-zero compression identities, and
the correspondence with the solver input and output. A saved success
report is not used in place of these checks. For example, from the
repository directory, the command
\begin{verbatim}
python3 code/srg.py verify --certificate 550_162_75_36
\end{verbatim}
checks the certificate giving the bound $506$.

The same entry point, \path{code/srg.py}, supports parameter inspection,
model construction, and new numerical searches. It accepts SRG parameters
whose selected spherical embedding is injective and has rational inner
products, with $2\le k\le\min(6,n)$ and nonnegative degree $d$.
For inputs with linearly dependent reference configurations, the
corresponding functions are fixed to zero; this restriction may weaken
the bound. All twelve applications in this paper have independent
reference configurations. New searches use SDPA-GMP. Its source distribution and the GMP source distribution are
included with build instructions and the original software notices.
The README files explain the supported inputs and commands.

The distributed certificates cover the four nonexistence results.
The larger experimental archive underlying Table~\ref{tab:computed},
the supplementary degree-two bounds, and the introductory coverage
count is retained by the authors and is outside this release.
\section{Recovering coefficient tables from polynomial bases}\label{app:degree}
Section~\ref{sec:framework} replaces each PSD coefficient polynomial by
its values on the labels used by the code constraints. This appendix
relates that coefficient-value description to a prescribed polynomial
basis. There are two questions. First, when does a chosen basis recover
every PSD coefficient table? Second, how large a monomial degree is
sufficient in the formulation of de Laat et al.? The answers also make
clear that coefficient-degree interpolation is different from lowering
the Gegenbauer degree bound $d$.

Let $W\subset\mathbb R^m$ be a finite set of distinct labels. For a
column $z(u)$ of polynomials, evaluation on $W$ gives a matrix $E$ whose
rows are $z(w)^T$. A PSD coefficient form $z(u)^THz(v)$ therefore gives
the table $EHE^T$. The following lemma characterizes when this produces
all PSD tables.

\subsection{When a polynomial basis captures all coefficient tables}

\begin{lemma}\label{lem:framework-features}
Let $z$ be any finite column vector of polynomials, and let $E$ have
rows $z(w)^T$, indexed by $w\in W$. Every PSD matrix on $W$ can be
written as $EHE^T$ with $H\succeq0$ if and only if $E$ has full row
rank.
\end{lemma}
\begin{proof}
If $E$ has full row rank, choose a matrix $C$ with $EC=I$; for example,
invert a square nonsingular submatrix and set the other rows of $C$ to
zero. Given $M\succeq0$, take $H=CMC^T\succeq0$. Then $EHE^T=M$.
Conversely, if the rows are dependent, choose $c\ne0$ with $E^Tc=0$.
Every matrix $EHE^T$ has zero quadratic form at $c$, whereas $I$ does
not. Thus not every PSD matrix can be represented.
\end{proof}

\subsection{A sufficient monomial degree}
We now specialize to the monomial basis used for the comparison with
de Laat et al. Let $z_q$ contain all monomials of total degree at most
$q$. If the components of $z$ have degree at most $q$, then
$z(u)^THz(v)$ has degree at most $q$ in each of $u,v$ separately and
joint total degree at most $2q$. In de Laat's degree parameterization,
the coefficient of the $l$th Gegenbauer term uses degree at most
$\delta-l$. As before, $m=|R|$, $s=|\Lambda|$ and $G=G_R$. The next
proposition gives a sufficient value of this coefficient degree.

\begin{proposition}\label{prop:framework-degree}
For $m>0$, indicators on $U_R$ can be chosen with total degree at most
$q_0=\max\{m(s-1),s\}$, and indicators on $U_R^{\mathrm{gen}}$ with
total degree at most $q_+=m(s-1)$. For $m=0$, take $q_0=q_+=0$.
Consequently, a model using all monomials of total degree at
most $\delta-l$ for the coefficient of $P_l^{n,m}$ gives exactly the finite function values
in Theorem~\ref{thm:framework-reduction}, using the same degrees $0,\ldots,d$,
whenever $\delta\ge\max\{q_0,d+q_+\}$.
\end{proposition}
\begin{proof}
For $w\in\Lambda^m$, define the grid indicator
\[
 p_w(u)=\prod_{j=1}^m\prod_{a\in\Lambda\setminus\{w_j\}}
                  \frac{u_j-a}{w_j-a}.
\]
Its degree is at most $m(s-1)$ and $p_w(v)=\delta_{wv}$ on
$\Lambda^m$, where $\delta_{wv}$ is $1$ when $w=v$ and $0$ otherwise.
For the $i$th reference label, put
\[
 q_i(u)=\prod_{a\in\Lambda}\frac{u_i-a}{1-a}.
\]
This polynomial has degree $s$, vanishes on $\Lambda^m$, and satisfies
$q_i(Ge_j)=\delta_{ij}$. The last identity holds because $(Ge_j)_i=1$
when $i=j$, and otherwise belongs to $\Lambda$.
Therefore
\[
 \phi_w(u)=p_w(u)-\sum_{i=1}^m p_w(Ge_i)q_i(u)
\]
still has the required generic-grid values and vanishes at every
reference label. The $\phi_w$ for feasible generic labels, together
with the $q_i$, are indicators on $U_R$ with degree at most $q_0$.
For positive degrees the unmodified $p_w$ already interpolate
$U_R^{\mathrm{gen}}$. Restricting to a feasible subset changes none of
these identities. The stated condition on $\delta$ provides these
polynomials for every coefficient, so Lemma~\ref{lem:framework-features}
gives the claimed equality of finite function values. Recall that
$G=BB^T$, with $B$ lower triangular and positive diagonal.
Changing from $u$ to $B^{-1}u$ is invertible and linear, and hence preserves each
space of polynomials of bounded total degree.
\end{proof}

\paragraph{The two-distance applications.}
For two inner products and references of size $m\le4$, the displayed
bounds give $q_0,q_+\le4$. Thus $\delta=9$ suffices for a monomial
representation of all coefficient tables used by our $d=5$ model,
provided the same Gegenbauer degrees $0,\ldots,5$ are retained. This
explains concretely how a sufficiently rich polynomial-basis model
recovers the coefficient-value reduction of Theorem~\ref{thm:framework-reduction}.
A monomial model that also includes degrees $6,\ldots,9$ is a different
search space. The subsequent degree-zero coordinate reduction preserves
the assembled dual inequalities, not every individual degree-zero
function value.

\subsection{Why the Gegenbauer degree bound is different}
The preceding interpolation concerns the coefficient polynomials in
$u,v$ only. It does not imply that higher Gegenbauer degrees can be
removed. The following example records this distinction.
\begin{remark}\label{rem:scalar-degree}
Take $n=3$, $R=\varnothing$ and $\Lambda=\{-1/2,1/2\}$. The PSD function
\[
 P_3^3(t)=\frac{5t^3-3t}{2}
\]
has values $P_3^3(1/2)=-7/16$ and $P_3^3(-1/2)=7/16$.
Any PSD polynomial function $F(t)$ of degree at most two has the form
$c_0P_0^3(t)+c_1P_1^3(t)+c_2P_2^3(t)$ with $c_i\ge0$.
Since $P_0^3,P_2^3$ are even and $P_1^3(t)=t$, it satisfies
\[
 F(1/2)-F(-1/2)=c_1\ge0,
\]
whereas the difference for $P_3^3$ is $-7/8$.
Thus no such PSD function agrees with $P_3^3$ at the allowed inner products,
although ordinary polynomial interpolation on $\Lambda\cup\{1\}$
uses degree at most two.
\end{remark}

\section{Continuity of the assembled dual function}\label{app:assembly-continuity}
We verify directly the joint continuity asserted after
\eqref{eq:local-assembly}. Recall that $V=S^{n-1}$ and that $I_{k-2}$
consists of spherical $\Lambda$-codes with at most $k-2$ points,
including the empty set. Suppose
$(x_j,y_j,Q_j)\to(x,y,Q)$ in $V^2\times I_{k-2}$.
The reference sizes are eventually equal, and their points can be
ordered to converge coordinatewise. Their Gram entries lie in the
finite set $\Lambda\cup\{1\}$, so eventually $Q_j$ has the same Gram
type as $Q$. Write $R$ for their common representative. The empty
reference causes no exception, since it is isolated.

For each $j$, let $\gamma_j=\gamma_{Q_j}$ be the chosen orthogonal
map carrying $Q_j$ to $R$. Every subsequence has a further subsequence
along which the orthogonal matrices $\gamma_j$ converge to an orthogonal
matrix $\gamma$: their entries are bounded, and the identity
$\gamma_j^T\gamma_j=I$ persists in the limit.
Since $\gamma_jQ_j=R$, coordinatewise convergence of the reference
points gives $\gamma Q=R$. Continuity of $K_R$ now gives
\[
 K_R(\gamma_jx_j,\gamma_jy_j)\longrightarrow K_R(\gamma x,\gamma y).
\]
The maps $\gamma$ and the chosen $\gamma_Q$ both carry $Q$ to $R$,
so the composition $\gamma\gamma_Q^{-1}$ preserves $R$. The symmetry
of $K_R$ therefore makes the limit equal to
$K_R(\gamma_Qx,\gamma_Qy)=T(x,y,Q)$.
If the original sequence of function values did not converge to
this value, some subsequence would stay a fixed positive distance
from it. The convergent subsequence just constructed would contradict
that property. Hence $T$ is jointly continuous, without a continuous
choice of the maps $\gamma_Q$.


\begin{thebibliography}{dLCdOV22}
\raggedright
\bibitem[BvM22]{BVM}
A.~E. Brouwer and H. Van Maldeghem,
\emph{Strongly Regular Graphs}, Encyclopedia of Mathematics and its
Applications 182, Cambridge University Press, 2022.


\bibitem[B26]{BrouwerTable}
A.~E. Brouwer, \emph{Parameters of Strongly Regular Graphs}, online tables.
\url{https://aeb.win.tue.nl/graphs/srg/srgtab.html}.
Accessed 28 September 2026. The authors retained the table snapshot
used to count the exclusions from the classical conditions.

\bibitem[DGS77]{DGS}
P. Delsarte, J.~M. Goethals, and J.~J. Seidel,
Spherical codes and designs,
\emph{Geometriae Dedicata} \textbf{6} (1977), 363--388.
\href{https://doi.org/10.1007/BF03187604}{doi:10.1007/BF03187604}.

\bibitem[M09]{MusinTwo}
O.~R. Musin, Spherical two-distance sets,
\emph{Journal of Combinatorial Theory, Series A} \textbf{116} (2009), 988--995.
\href{https://doi.org/10.1016/j.jcta.2008.09.003}{doi:10.1016/j.jcta.2008.09.003}.

\bibitem[KG22]{KG}
J.~H. Koolen and B. Gebremichel,
There does not exist a strongly regular graph with parameters
$(1911,270,105,27)$,
\emph{Electronic Journal of Combinatorics} \textbf{29}(2) (2022), Paper P2.3.
\href{https://doi.org/10.37236/10675}{doi:10.37236/10675}.

\bibitem[dLCdOV22]{dL}
D. de Laat, F. Caluza Machado, F.~M. de Oliveira Filho, and F. Vallentin,
$k$-Point semidefinite programming bounds for equiangular lines,
\emph{Mathematical Programming} \textbf{194} (2022), 533--567.
\href{https://doi.org/10.1007/s10107-021-01638-x}{doi:10.1007/s10107-021-01638-x}.
Earlier version: \href{https://arxiv.org/abs/1812.06045v2}{arXiv:1812.06045v2}, 2019.
Numbered references use the published version unless stated otherwise.

\bibitem[M14]{Musin}
O.~R. Musin, Multivariate positive definite functions on spheres,
in \emph{Discrete Geometry and Algebraic Combinatorics},
Contemporary Mathematics \textbf{625} (2014), 177--190.
\href{https://arxiv.org/abs/math/0701083v2}{arXiv:math/0701083v2}, 2008.
Definition and theorem numbers refer to this preprint version.

\bibitem[KY22]{KY}
W.-J. Kao and W.-H. Yu,
Four-point semidefinite bound for equiangular lines,
\href{https://arxiv.org/abs/2203.05828v1}{arXiv:2203.05828v1}, 2022.

\bibitem[N10]{SDPA}
M. Nakata, A numerical evaluation of highly accurate multiple-precision
arithmetic version of semidefinite programming solver: SDPA-GMP, -QD
and -DD, in
\emph{2010 IEEE International Symposium on Computer-Aided Control System
Design}, 29--34.
\href{https://doi.org/10.1109/CACSD.2010.5612693}{doi:10.1109/CACSD.2010.5612693}.

\bibitem[DLMF]{DLMF}
National Institute of Standards and Technology,
\emph{Digital Library of Mathematical Functions},
Section~18.5, Eq.~18.5.10.
\url{https://dlmf.nist.gov/18.5.E10}.

\end{thebibliography}
\end{document}